\documentclass{amsart}

\usepackage{amssymb}
\usepackage{amsmath}
\usepackage{mathrsfs}
\usepackage[english]{babel}
\usepackage[utf8]{inputenc}
\usepackage[T1]{fontenc}
\usepackage[dvipsnames]{xcolor}
\usepackage{dsfont}
\usepackage{lipsum} 

\usepackage[top=3cm,bottom=2cm,left=3cm,right=3cm,marginparwidth=1.75cm]{geometry}

\usepackage{graphicx}
\usepackage{subcaption}
\usepackage[colorlinks=true, allcolors=OliveGreen, pagebackref=true]{hyperref}

\usepackage{cite}
\usepackage{enumitem}
\usepackage[capitalise]{cleveref}
\usepackage{autonum}

\newcommand{\bbC}{\mathbb{C}}		
		
\newcommand{\bbE}{\mathbb{E}}

\newcommand{\bbN}{\mathbb{N}}		
		
\newcommand{\bbP}{\mathbb{P}}		
		
\newcommand{\bbR}{\mathbb{R}}

\newcommand{\bbZ}{\mathbb{Z}}		

\newcommand{\scrF}{\mathscr{F}}

\newcommand{\scrS}{\mathscr{S}}

\newcommand{\calE}{\mathcal{E}}

\newcommand{\calG}{\mathcal{G}}

\newcommand{\calK}{\mathcal{K}}

\newcommand{\calP}{\mathcal{P}}

\newcommand{\calS}{\mathcal{S}}

\newcommand{\calX}{\mathcal{X}}

\newcommand{\calZ}{\mathcal{Z}}

\newcommand{\mrE}{\mathrm{E}}

\newcommand{\mrK}{\mathrm{K}}

\newcommand{\mrd}{\mathrm{d}}

\newcommand{\seg}{\underline}

\newcommand{\set}[2]{\left\{ #1 \ | \ #2\right\}}
\newcommand{\ball}{b}
\newcommand{\mathand}{\quad \text{and} \quad}
\newcommand{\del}{\partial}
\newcommand{\delbar}{\bar{\partial}}

\newcommand{\sol}[1]{\calS\left\{#1\right\}}
\newcommand{\esol}[1]{\mathbb{E}\mathbb{S}\left\{#1\right\}}
\newcommand{\Rpos}{\bbR_{>0}}
\newcommand{\Exp}{\calE}
\newcommand{\GRF}{\calX}
\newcommand{\cov}{\mrK}

\DeclareMathOperator{\vol}{vol}
\DeclareMathOperator{\Span}{Span}

\DeclareMathOperator{\MV}{V}
\DeclareMathOperator{\conv}{conv}

\DeclareMathOperator{\interior}{int}
\DeclareMathOperator{\diam}{diam}
\DeclareMathOperator{\Hom}{Hom}
\DeclareMathOperator{\dom}{dom}

\def\randin{%
  \mathchoice%
  {\raisebox{-.35ex}{$\displaystyle{^\subset}$}\mkern-11.5mu\raisebox{+.45ex}{$\displaystyle{_\subset}$}}
  {\mkern+1mu\raisebox{-.27ex}{$\textstyle{^\subset}$}\mkern-11.7mu\raisebox{+.45ex}{$\textstyle{_\subset}$}}
  {\raisebox{.35ex}{$\scriptstyle\subset$}\mkern-14mu\raisebox{-.15ex}{$\scriptstyle\subset$}}
  {\raisebox{.3ex}{$\scriptscriptstyle\subset$}\mkern-13.5mu\raisebox{-.10ex}{$\scriptscriptstyle\subset$}}
}

\newtheorem{theorem}{Theorem}[section]
\newtheorem*{theorem*}{Theorem}
\newtheorem{proposition}[theorem]{Proposition}
\newtheorem{corollary}[theorem]{Corollary}
\newtheorem{lemma}[theorem]{Lemma}
\newtheorem{mainthm}{Theorem}

\theoremstyle{definition} 
\newtheorem{definition}[theorem]{Definition}

\newtheorem{example}[theorem]{Example}
\theoremstyle{remark} 
\newtheorem{remark}[theorem]{Remark}

\numberwithin{equation}{section}
\numberwithin{theorem}{section}

\crefname{equation}{}{}

\title{Real Gaussian exponential sums via a real moment map}
\author{Léo Mathis }
\address{Institut für Mathematik, 
Goethe-Universität Frankfurt, Germany}
\email{mathis@mathematik.uni-frankfurt}
\thanks{Partially funded by DFG grant BE 2484/10-1}
\begin{document}


\begin{abstract}
    We study the expected number of solutions of a system of identically distributed exponential sums with centered Gaussian coefficient and arbitrary variance. We use the Adler and Taylor theory of Gaussian random fields to identify a moment map which allows to express the expected number of solution as an integral over the Newton polytope, in analogy with the Bernstein Khovanskii Kushnirenko Theorem. We apply this result to study the monotonicity of the expected number of solution with respect to the support of the exponential sum in an open set. We find that, when a point is added in the support in the interior of the Newton polytope there exists an open sets where the expected number of solutions decreases, answering negatively to a local version of a conjecture by Bürgisser. When the point added in the support is far enough away from the Newton polytope we show that
    there is an unbounded open set where the number of solution decreases. We also prove some new lower bounds for the Aronszajn multiplication of exponential sums.
\end{abstract}

\maketitle

\section{Introduction}
From a finite subset $A\subset \bbZ^m$, one can construct a family of (Laurent) polynomials in the variable $w=(w_1,\ldots,w_m)$:
\begin{equation}\label{eq:intro PA}
  Q(w):=\sum_{a\in A}c_a w^a
\end{equation} 
where $w^a:=w_1^{a_1}\cdots w^{a_m}$. Such polynomials are often called \emph{sparse}. Over the complex numbers, the main result in the study of roots of sparse polynomials is the Berstein Kushnirenko Khovanskii (BKK) Theorem that determines the number of solutions of a system of polynomials of the form \cref{eq:intro PA}. Before stating it, let us introduce the following notation. For every tuple of functions $f_1:M\to \bbR^{k_1},\ldots,f_d:M\to \bbR^{k_d}$ where $M$ is a smooth manifold (in that case $\bbC^m$) and every open set $U\subset M$, we denote the set of solutions of the corresponding system by:
\begin{equation}
  \sol{f_1,\ldots,f_d; U}:=\set{x\in U}{f_{1}(x)=\cdots= f_{d}(x)=0}.
\end{equation}
In its most simple form, the BKK Theorem says the following, see for example \cite{Bernshtein1975}.
\begin{theorem*}[BKK]
  Let $A\subset \bbZ^m$ be a finite subset and for all $i=1,\cdots, m$, let $Q^{i}(w):=\sum_{a\in A}c^i_a w^a$ with $c_a^i\in \bbC$. Then for a generic choice of $c_a^i$, we have 
  \begin{equation}\label{eq:BKK}
    \#\sol{Q^1,\ldots, Q^m;(\bbC^*)^m}=m! \vol_m (P)
  \end{equation}
  where $P:=\conv(A)\subset \bbR^m$ is the convex hull of $A$, also called the \emph{Newton polytope}.
\end{theorem*}
Here the word \emph{generic} means that in the space of coefficients $(\bbC^A)^m$, Equation \cref{eq:BKK} is true except on a set of measure zero. In that case, this set is the complex subvariety of complex codimension one where the system \cref{eq:BKK} is degenerate, see \cite[Chapter~I-4]{SchubSmaleBezoutI}.

In this paper we want to study \emph{real} solutions of systems of \emph{real} polynomials. In that situation, there is no more \emph{generic case}, indeed the subset of degenerate systems divide the space of coefficient in different connected components where the number of solution is different. It is then possible to take a probabilistic approach and study the \emph{average} number of solutions of a \emph{random} system. We establish in this context a result analogous to the BKK Theorem, but where the volume of the Newton polytope has to be taken with respect to a certain Riemannian metric, see \cref{mainthm: real BKK} below. 

\subsection{Main results}
Before describing the result, we adopt the point of view of \emph{exponential sums}. For every $w\in \bbR^m_{>0}=\set{w\in \bbR^m}{w_i>0,\,\forall i=1,\ldots, m}$ in the positive orthant we write $x_i:=\log(w_i)$ in such a way that, for every $a\in \bbN^m$, we have $w^a=e^{\langle a, x\rangle}$. This allows to take exponents that are reals and not only integers. Sums of monomials with real exponents are also called \emph{signomials} and their optimization have application in physics and chemistry, see for example \cite{DresderSignom} and the references therein. It turns out to be equivalent and more convenient to work with exponential sum, see the discussion in \cref{sec: poly}.

For every finite set $A\subset\bbR^m$ and coefficients $(\alpha_a)_{a\in A}$ with $\alpha_a>0$, we define the \emph{Gaussian exponential sum}:
\begin{equation}\label{eq: def expsum in intro}
  \Exp(x):=\sum_{a\in A}\xi_a\,\alpha_a\, e^{\langle a, x\rangle}
\end{equation}
where $\xi_a\in \bbR$ are iid standard Gaussian variables. We introduce another notation. If $\GRF:M\to\bbR^k$ is a random function on a smooth manifold $M$ of dimension $m=k n$ (in our case $M=\bbR^m$ and $k=1$) and $U\subset M$ an open set, we write the (possibly infinite) expected number of solutions in $U$ as:
\begin{equation}
  \esol{\GRF;U}:=\bbE\#\sol{\GRF_1,\ldots,\GRF_n; U}
\end{equation}
where $\GRF_1,\ldots,\GRF_n$ are iid, distributed as $\GRF$ .

In this article, we will show how the relevant information for the expected zeros of \cref{eq: def expsum in intro} is contained in the \emph{potential} $\Phi:\bbR^m\to \bbR$ given for all $x\in \bbR^m$ by:
\begin{equation}
  \Phi(x):=\frac{1}{2}\log\left(\bbE\left[\Exp(x)^2\right]\right).
\end{equation}
The function $\cov(x):=\bbE\left[\Exp(x)^2\right]=\sum_{a\in A}\alpha_a^2 e^{2\langle a, x\rangle}$ is called the \emph{diagonal covariance} function of the Gaussian random field $\Exp$. We prove in \cref{lem:Phiconv} that $\Phi$ is strictly convex. We can then consider its \emph{Legendre transform} $\Phi^*$, see \cref{eq:def Legtrans}, which turns out to be a (smooth strictly convex) function on the Newton polytope $P$. We then obtain the following which is \cref{thm:esol is volg P}.

\begin{mainthm}\label{mainthm: real BKK}
Let $\Exp$ be given by \cref{eq: def expsum in intro}. Then we have:
\begin{equation}
  \esol{\Exp;\bbR^m}=\frac{1}{2^{\tfrac{m-2}{2}}s_m} \vol_g (P)
\end{equation}
where $s_m$ denotes the Euclidean volume of the unit sphere $S^m\subset \bbR^{m+1}$ and $\vol_g$ is the Riemannian volume for the metric $g:=D^2_p\Phi^*$, i.e., $g$ is given by the Hessian of the Legendre transform $\Phi^*$ of the potential $\Phi$.
\end{mainthm}
Note that the total number of solutions $\esol{\Exp;\bbR^m}$ is known to be finite even if $A$ is not rational. For example, an upper bound in terms of the cardinality $\#A$ was recently obtained by Ergür, Telek and Tonelli-Cueto in \cite{TonCuetNRZRSP}.

\cref{mainthm: real BKK} is obtained by identifying a \emph{moment map} $\mu:=D_x\Phi$ in this context. As the name suggest, it is similar to the moment map in toric geometry, see \cite[Section~4.2]{FultonToric} or in the complex setting, see \cite{KlartagMoment}. We prove in \cref{lem:muiso} that 
\begin{equation}
  \mu:\bbR^m\to \interior(P)
\end{equation}
is a diffeomorphism between $\bbR^m$ and the interior of the Newton polytope.

By studying the potential $\Phi$ and by some elementary properties of Legendre transform, we obtain, for example, in \cref{thm:lowerbound}, a new lower bound for the expected number of solution:
\begin{equation}
  \esol{\Exp;\bbR^m}> \frac{1}{2^{m-1}s_m} \,\frac{\vol_m(P)}{\diam(P)^m}
\end{equation}
where $\vol_m(P)$ is the classical Euclidean volume of the Newton polytope and $\diam(P)$ its diameter.

The main results we obtain are to study the \emph{monotonicity} of Gaussian exponential sums. More precisely, we chose $a_0\in \bbR^m\setminus A$ and $\alpha_0>0$ and we want to compare the zeros of the new exponential sum 
\begin{equation}
  \Exp_0(x):=\Exp(x)+\xi_0\,\alpha_0\, e^{\langle a_0, x\rangle}
\end{equation}
where $\xi_0\in \bbR$ is a standard Gaussian variable independent of $\Exp$, with the zeros of $\Exp$. In the following we always assume that the Newton polytope $P$ is full dimensional.

Peter Bürgisser conjectured (this conjecture only appears on the Bachelor thesis of Hendrik Schroeder \cite{SchroederBA}) that, if $a_0$ is in the interior of the Newton polytope $P$, then $\esol{\Exp_0;\bbR^m}\geq \esol{\Exp;\bbR^m}$ . Here we are able to prove that a local version of Bürgisser's conjecture does not hold. The following is \cref{thm:a0 in int} in the text.
\begin{mainthm}\label{mainthm:aoinP}
  Let $a_0\in \interior(P)$, then there is a non empty open set $U\subset\bbR^m$ such that
  \begin{equation}
    \esol{\Exp_0;U}< \esol{\Exp;U}.
  \end{equation}
\end{mainthm}

In the case where the distance $d(a_0,P)$ from $a_0$ to the Newton polytope is big enough, we obtain the following which, in the text, corresponds to \cref{thm: a0 out}. 
\begin{mainthm}\label{mainthm:aooutP}
  Let $a_0$ be such that $d(a_0,P)>\diam(P)/m$. There exists a non empty unbounded open set $U\subset \bbR^m$ such that 
    \begin{equation}
      \esol{\Exp_0;U}<\esol{\Exp;U}.
    \end{equation}
\end{mainthm}

\subsection{Methods}
To prove our results, we take a step back in abstraction and study Gaussian random fields (GRF) on manifolds. A GRF on a smooth manifold $M$ is a random function $\GRF:M\to \bbR$, in our case smooth, such that the evaluation at tuples of points is a Gaussian vector.

Adler and Taylor developed a theory based on the Kac-Rice formula, see \cref{sec:KR}, to explain how the expected number of solutions $\esol{\GRF;\cdot}$ is in fact the Riemannian volume for a certain Riemannian metric $g$, that we call the Adler-Taylor (AT) metric, see \cref{prop:esolisATvol}. The AT metric is described in terms of the GRF $\GRF$ in \cref{prop:ATmetric}.

We are interested in the case where the GRF is contained in a finite dimensional subspace of $C^\infty(M)$, that is, when it is of the form 
\begin{equation}
  \GRF=\xi_1 \, f_1 +\cdots +\xi_d \, f_d
\end{equation}
where $f_1,\ldots,f_d\in C^\infty(M)$ are fixed smooth functions having no common zero, see \cref{def:admissible}. The AT metric can then be expressed in terms of the functions $f_i$ and the moment map $\mu$, see \cref{eq:g_x in basis} .

We then chose $f_0\in C^\infty(M)$ linearly independent of $f_1,\ldots,f_d$ and consider the GRF $\GRF_0:=\GRF+\xi_0\,f_0$ where $\xi_0\in\bbR$ is a standard Gaussian variable independent of $\GRF$. We then introduce the function $\Psi:M\to \bbR$ given for all $x\in M$ by 
\begin{equation}
  \Psi(x):=\left(\frac{\cov(x)}{\cov(x)+f_0(x)^2}\right)^{\frac{m}{2}}\sqrt{1+ g^x(\tau(x))}.
\end{equation}
where $\tau$ is a one form on $M$ defined in terms of $f_0$ and the moment map $\mu$ and $g^x$ is the AT metric for $\GRF$ on forms, see \cref{def:tau Psi}. We prove that this function determines the local monotonicity, see \cref{thm:Psi} in the text.

\begin{mainthm}
  For every open set $U\subset M$ we have 
  \begin{equation}
    \esol{\GRF_0;U}-\esol{\GRF;U}=c_m\int_U (\Psi(x)-1) \, \mrd M (x)
  \end{equation}
  where $c_m>0$ is a dimensional constant and $\mrd M (x)$ denotes integration with respect to the AT metric of $\GRF$.
\end{mainthm}
It follows that, to determine the local monotonicity, one must determine the open sets $U_-:=\{\Psi<1\}$ and $U_+:=\{\Psi>1\}$. In \cref{sec:monot exp sums}, we study this in the case of Gaussian exponential sum to prove \cref{mainthm:aoinP} and \ref{mainthm:aooutP}.

Interestingly enough, we show in \cref{prop:alpha with Phio} that the one form $\tau$ is integrable, that is that it is orthogonal to a family of hypersurfaces. We show in \cref{prop:decreasHypGRF} that, when restricted to these hypersurfaces, the number of solution always decreases. In the case of exponential sums, this is described in \cref{prop:decreasHypExp}.

\subsection{Other results}

We also study two operations on Gaussian exponential sums. Firstly a simple tensor product, see \cref{eq:tensprodExp}. The second operation, more subtle, was introduced by Malajovitch in \cite{MALAJO} and is called \emph{Aronszajn's multiplication}. Given two finite sets $A,B\subset \bbR^m$ and choices of coefficients $\alpha_a,\beta_b>0$, we consider the corresponding exponential sums $\Exp_\alpha$ and $\Exp_\beta$. Their Aronszajn multiplication, denoted $\Exp_\alpha \odot \Exp_\beta$ is the exponential sum with support the Minkowski sum $C:= A+B$ and coefficients $\gamma_c>0$ given for all $c\in C$ by:
\begin{equation}
  \gamma_c:=\sqrt{\sum_{a+b=c}\alpha_a^2\beta_b^2}.
\end{equation}

It turns out that the AT metric corresponding to the tensor product, respectively to the Aronszajn multiplication, is given by the direct sum, respectively the sum, of the previous AT metrics, see \cref{lem:Aronsz and AT}. We deduce how the number of solutions transform, see \cref{prop:esol of tens} for the tensor product and \cref{lem:Aronsz and AT} for the powers of Aronszajn multiplication, which was already obtained in \cite{MALAJO} and reproved here in our slightly more general context of admissible fields. Finally, Malajovitch proved that the expected number of solutions is sub additive with respect to Aronszajn multiplication. We reprove this result with the understanding of the AT metrics involved and give a new lower bound, see \cref{thm:uplowAron} for the more general statement.
\begin{mainthm}
Let $\Exp, \Exp'$ be Gaussian exponential sum and let $\Exp\odot\Exp'$ be their Aronszajn multiplication. For any tuple of independent Gaussian exponential sum $\Exp_1,\ldots,\Exp_{m-1}$, write $\mrE:=(\Exp_1,\ldots,\Exp_{m-1}):\bbR^m\to\bbR^{m-1}$. Then, for every open set $U\subset \bbR^m$, we have:
\begin{equation}
  \frac{1}{\sqrt{2}}\left(\bbE\#\sol{\Exp,\mrE;U}+\bbE\#\sol{\Exp',\mrE;U}\right) \leq \bbE\#\sol{\Exp\odot\Exp',\mrE;U} \leq \bbE\#\sol{\Exp,\mrE;U}+\bbE\#\sol{\Exp',\mrE;U}.
\end{equation}
\end{mainthm}

We illustrate these operations in an explicit example in \cref{sec:example} where we also compute the open set $U_-$.

Finally, to complete the picture we show in \cref{sec:Cpx} how to prove the BKK theorem in our context. More precisely, we proved a generalization by Kazarnovskii in \cite{Kazarnovskii2018MixedHV}. The proof turns out to be similar to the one given by Kazarnovskii. However it is given in more details and offer a perspective towards a generalization, for example to quaternions.

If $A\subset \bbC^n$ is a finite subset, we define the complex exponential sum
\begin{equation}
  \Exp^\bbC(z):=\sum_{a\in A}\xi_a e^{(a,z)}
\end{equation}
with $\xi_a\in\bbC$ complex Gaussian and $(\cdot,\cdot)$ denoting the standard Hermitian form. Then we prove the following which is \cref{thm:BKKprob} in the text.

\begin{mainthm}\label{mainthm:Kaza}
  For all $U\subset \bbC^n$, we have 
  \begin{equation}
    \esol{\Exp^\bbC;U}=\frac{n!}{\pi^n}\int_U \det\left(\del_z\delbar_z\Phi^\bbC\right)\, \mrd z.
  \end{equation}
  where $\Phi^\bbC(z):=\log\left(\sum_{a\in A} e^{2\Re{(a,z)}}\right) $.
\end{mainthm}

To see how this result implies the BKK theorem, one needs to take $A\subset \bbR^n$ and apply the change of variable given by the moment map. See more details in the proof of \cref{thm:BKKprob} below.

\subsection*{Acknowledgments} The author wishes to thank Bernd Sturmfels and Josue Tonelli Cueto for giving the spark(s) that initiated this project. We also thank Peter Bürgisser and Andreas Bernig for helpful discussions and Emanuel Milman for some interesting comments.






\section{Preliminaries}

We use the following convenient notation:
\begin{itemize}
  \item for every vector $x$ in a vector space we write 
  \begin{equation}\label{eq:seg}
    \seg{x}:=\tfrac{1}{2}[-x,x]
  \end{equation}
  \item if $\scrS$ is a measurable set and $S$ is a random element of $\scrS$, we write 
  \begin{equation}
    S\randin \scrS.
  \end{equation}
\end{itemize}

In what follows, $V$ denotes a real vector space of dimension $m<+\infty$ and $V^*:=\Hom_\bbR(V,\bbR)$ its dual.

\subsection{Convex geometry}\label{sec:Convex}
We start with a few standard facts on convex geometry. For more details, we refer to \cite{bible}. A \emph{convex body} $K$ in $V$ is a compact convex non empty subset of $V$. Its \emph{support function} is the function $h_K: V^*\to \bbR$ given for all $u\in V^*$ by: 
\begin{equation}\label{eq:defhK}
  h_K(u):=\sup\set{\langle u, x\rangle}{x\in K}
\end{equation}
The support function determines the convex body $K$, meaning that two convex bodies $K$ and $K'$ are equal if and only if $h_{K}=h_{K'}$, see \cite[Section~1.7.1]{bible}. Moreover, a function $h:V^*\to \bbR$ is the support function of a convex body in $V$ if and only if it is \emph{sublinear}, that is if $h(\lambda u)=\lambda h(u)$ for all $u\in V^*, \lambda \geq 0$ and $h(u+v)\leq h(u)+h(v)$ for all $u,v\in V^*$; see \cite[Theorem~1.7.1]{bible}.

If $K$ and $K'$ are two convex bodies, their \emph{Minkowski sum} is the convex body
\begin{equation}
  K+K':=\set{x+y}{x\in K,\, y\in K'}.
\end{equation}
If $V$ is Euclidean, the \emph{Hausdorff distance} between $K$ and $K'$ is given by (see \cite[ Lemma 1.8.14]{bible}) 
\begin{equation}
  d(K,K')=\sup\set{|h_K(u)-h_{K'}(u)|}{\|u\|=1}.
\end{equation}
The topology it induces on the space of convex bodies does not depend on the choice of euclidean structure. 
Moreover, we note that the support function is \emph{linear} with respect to Minkowski sum, that is $h_{K+K'}=h_K+h_{K'}$. Moreover, if $L:V\to W$ is a linear map to a vector space $W$, 
then $h_{L(K)}=h_K\circ L^T$ where $L^T:W^*\to V^*$ is the \emph{transpose} of $L$, recall that it is given, fro all $u\in W^*$ and $v\in V$ by 
\begin{equation}\label{def:LT}
  \langle L^T u, v\rangle:=\langle u, L v\rangle.
\end{equation}
Another direct consequence of the definition is that for convex bodies $K,K'$ we have
\begin{equation}\label{eq:Incl on supp}
  K\subset K' \quad \iff \quad h_K(u)\leq h_{K'}(u) \, \forall u\in V^*.
\end{equation}

The \emph{convex hull} of a bounded set $A\subset V$ is the smallest convex body containing $A$ and is denoted $\conv(A)$. We note that, for its support function $h_{\conv(A)}$, the $\sup$ in \eqref{eq:defhK} can be taken over $A$ instead of $\conv(A)$. The convex hull of a finite set is called a \emph{polytope}.

For all $u\in V^*$ and compact set $A\subset V$ we denote by 
\begin{equation}
  A^u:=\set{a\in A}{\langle u, a\rangle=h_{\conv(A)}(u)}=\set{a\in A}{\langle u, a\rangle=\sup\set{\langle u, a'\rangle}{a'\in A}}.
\end{equation}
If $A=K$ is a convex body, $K^u$ is also called the \emph{exposed face} of $K$ in the direction $u$, we will simply call it a \emph{face}. 

A \emph{vertex} $v\in P$ of a polytope $P$ is a face of dimension zero, i.e., $v$ is a vertex if and only if there is $x\in \bbR^m$ such that $P^x=\{v\}$. Note that if $P=\conv(A)$ then all vertices belong to $A$.

\begin{remark}\label{rk:NC vertex dense}
  By a simple dimensional count, one can also see that the set $\set{x\in\bbR^m}{\dim(P^x)=0}$ is dense in $\bbR^m$.
\end{remark}

We now fix a Euclidean structure on $V$. The space of convex bodies, that we denote $\calK(V)$, is endowed with the topology induced from the Hausdorff distance. This is equivalent to consider the supremum norm of support functions as functions on the sphere, see \cite[Lemma1.8.14 ]{bible}. Recall that the \emph{mixed volume} is defined as follows, see \cite[Theorem~5.1.7]{bible}. We write $\calK(V)$ for the space of convex bodies of $V$ (endowed with the Hausdorff distance).

\begin{definition}
The \emph{mixed volume} is the unique continuous application $\MV:\calK(V)^m\to \bbR_{\geq0}$ that is symmetric and Minkowski linear in each of its argument and that is such that for all $K_1,\ldots,K_n\in \calK(V)$ and $t_1,\ldots,t_n\geq 0$ we have 
\begin{equation}
  \vol(t_1K_1+\cdots+t_n K_n)=\sum_{i_1,\ldots,i_m=1}^nt_{i_1}\cdots t_{i_m}\MV(K_{i_1},\ldots,K_{i_m}).
\end{equation} 

\end{definition}

In the following, we write
\begin{equation}
  \|K\|:=\sup\set{\|x\|}{x\in K}.
\end{equation}
Note that $\|K\|$ is the smallest number such that $K\subset \|K\|\, B_V$ where $B_V$ is the unit ball of $V$. Equivalently $h_K(u)\leq \|K\|\|u\|$ for all $u\in V^*$.

We now describe how to obtain convex bodies with measures.

\begin{definition}\label{def:Minkowski integral}
  Let $\set{K_s}{s\in S}$ be a measurable family of convex bodies in $V$ indexed by a measurable space $S$. Let $\mu$ be a non negative measure on $S$ such that $\int_S \|K_s\|\,\mrd \mu(s)<\infty $ for one (and thus any) Euclidean structure on $V$. Then, the \emph{Minkowski integral} is the convex body $\int_S K_s \,\mrd\mu(s)$ whose support function is given for all $u\in V^*$ by:
  \begin{equation}
    h_{\int_S K_s \mrd\mu(s)}(u):=\int_S h_{K_s}(u) \,\mrd\mu(s).
  \end{equation}
  In the particular case where $\mu$ is a probability measure, we adopt the probabilistic notation. That is, if $K$ is a random convex body of $V$ with $\bbE\|K\|<\infty$, we write $\bbE K$ for the convex body with support function given for all $u\in V^*$ by:
  \begin{equation}
    h_{\bbE K}(u):=\bbE h_{K}(u) .
  \end{equation}
\end{definition}

In the particular case where the random convex bodies are segments we obtain a particular class of convex bodies called \emph{zonoids} see \cite[Theorem~3.1]{VitaleEARDZ}. Recall the notation for segments \cref{eq:seg}.

\begin{definition}\label{def:zonoids}
  A convex body $K\subset V$ is a \emph{zonoid} if there is a random vector $X\randin V$ with $\bbE\|X\|<+\infty$ such that $K=\bbE\seg{X}$. We call $\bbE\seg{X}$ the \emph{Vitale zonoid} associated to $X$.
\end{definition}

\begin{remark}
  Translates of such bodies are usually also called zonoids and the bodies defined in \cref{def:zonoids} are called \emph{centered} zonoids which refers to the fact that their center of symmetry is at the origin. In this paper we only consider \emph{centered} zonoids thus we shall continue to simply call them zonoids.
\end{remark}

\begin{remark}
  If $\mu$ is a finite measure on $V$ such that $\int_V\|x\| \,d\mu(x)<+\infty$ then $K=\int_V\seg{x}\, d\mu(x)$ is also a zonoid. Indeed if $\tilde{X}\randin V$ is a random vector with probability distribution $\mu/\mu(V)$ and $X:=\mu(V) \tilde{X}$, then one can see, comparing support functions, that $K=\bbE\seg{X}$.
\end{remark}
Concretely, if $X\randin V$ is a random vector such that $\bbE\|X\|<+\infty$ then the support function of the associated Vitale zonoid is given for all $u\in V^* $ by:
\begin{equation}\label{eq:sptfct Vitale}
  h_{\bbE\seg{X}}(u)=\frac{1}{2}\bbE|\langle u, X\rangle|.
\end{equation}
We note that, if $L:V\to W$ is a linear map, we have 
\begin{equation}\label{eq:Vitale linear}
  \bbE\seg{L(X)}=L\bbE\seg{X}.
\end{equation} 

We will, in the most general case, need to consider zonoids in the exterior algebra $\Lambda^k V$. Recall that, if $V$ is Euclidean, that is, has a scalar product $\langle\cdot,\cdot\rangle$, then the exterior power $\lambda^kV$ also inherits a scalar product that is given, for all $x_1,y_1,\ldots,x_k,y_k\in V$, by:

\begin{equation}
  \langle x_1\wedge \cdots\wedge x_k, y_1\wedge \cdots \wedge y_k\rangle:=\det(\langle x_i, y_j\rangle)_{1\leq i,j\leq k}.
\end{equation}

Let us write $\calZ(V)$ the space of zonoids in $V$. Then by \cite[Theorem 4.1]{ZonAlg}, there is a unique bilinear continuous operation $\cdot\wedge\cdot:\calZ(\lambda^kV)\times \calZ(\Lambda^l V)\to \calZ(\Lambda^{k+l}V)$ such that, if $X\randin\Lambda^kV$ and $Y\randin \Lambda^l V$ are independent random vectors such that $\bbE\|X\|\,,\bbE\|Y\|<\infty$, we have:
\begin{equation}
  (\bbE\seg{X})\wedge (\bbE\seg{Y})=\bbE\seg{X\wedge Y}.
\end{equation}

We will also write , in accordance to \cite[Theorem 5.2]{ZonAlg}, $\ell:\calK(V)\to \bbR$ the first intrinsic volume. For zonoids, if $X\randin V$ is such that $\bbE\|X\|<\infty$ then we have, see \cite[Definition 2.10]{ZonAlg}:
\begin{equation}
  \ell(\bbE\seg{X})=\bbE\|X\|.
\end{equation}

This is a monotonically increasing function, i.e., if $K\subset K'$ then $\ell(K)\leq \ell(K')$, see \cite[Corollary~2.13]{ZonAlg}. Moreover the wedge product is also increasing, see \cite[Proposition 3.4]{ZonAlg}.

\begin{proposition}
  Let $K\subset K'\subset \Lambda^k V$ and $L\subset \Lambda^lV$ be zonoids. Then we have $K\wedge L\subset K'\wedge L$.
\end{proposition}

\subsection{Legendre transform}

We gather here some properties of Legendre transformations, the reader can refer to \cite[Section~12 and 26]{Rockafellar}.
Let $\Phi: \bbR^m\to \bbR\cup\{+\infty\}$ be a lower semi continuous convex function. The \emph{domain} of $\Phi$ is the closure of the set where it is finite:
\begin{equation}\label{eq:def Legtrans}
\dom(\Phi):=\overline{\set{x\in \bbR^m}{\Phi(x)<+\infty}}
\end{equation}
One easily check that the convexity of $\Phi$ implies that $\dom()\Phi)$ is convex.

The \emph{Legendre transform} of $\Phi$ is the function $\Phi^*;\bbR^m \to \bbR\cup\{+\infty\}$ given for all $p\in \bbR^m$ by:
\begin{equation}\label{eq:def Leg trans}
  \Phi^*(p):=\sup\set{\langle p, x \rangle - \Phi(x)}{x\in \bbR^m}
\end{equation}

It follows, see \cite[Theorem~12.2]{Rockafellar}, that $\Phi^*$ is a lower semi continuous convex function. If $\Phi$ is $C^1$, then the supremum in \cref{eq:def Leg trans}, if attained, is attained at a point $x^*$ such that $D_{x^*}\Phi=p$ where $D_{x^*}\Phi$ is the differential (i.e., in that case the gradient) of $\Phi$ in $x^*$. In other words, we have 
\begin{equation}\label{eq:legtr grad}
\Phi^*(D_x\Phi)=\langle D_x\Phi,x\rangle-\Phi(x).
\end{equation}
If this supremum is not attained then, we have $\Phi^*(p)=+\infty$.
We obtain the following.

\begin{lemma}\label{prop:LT C1}
  Let $\Phi:\bbR^m\to \bbR$ be a $C^1$ convex function. Then the image of $D\Phi:\bbR^m\to \bbR^m;x\mapsto D_x\Phi$ is the relative interior of $\dom(\Phi^*)$. In particular the image of $D\Phi$ is convex.
\end{lemma}

\begin{proof}
  The relative interior $U:=\mathrm{relint}(\dom (\Phi^*))$ corresponds to the points where $\Phi^*$ is finite, i.e. where the supremum \cref{eq:def Leg trans} is attained. By \cref{eq:legtr grad}, it means that $p\in U$ if and only if there is $x\in\bbR^m$ such that $D_x\Phi$, i.e., if and only if $p$ is in the image of $D\Phi$.
\end{proof}

\begin{proposition}\label{prop:LT C2}
  Let $\Phi:\bbR^m\to \bbR$ be a $C^2$ strictly convex function. Then $\Phi^*$ is $C^2$ and strictly convex on $\interior(\dom(\Phi^*))$. In that case, the inverse of $D\Phi:\bbR^m\to \interior(\dom (\Phi^*))$ is $D\Phi^{-1}=D\Phi^*$.
\end{proposition}

\begin{proof}
  The function $\Phi$ is strictly convex if and only if, for all $x\in \bbR^m$, the Hessian $D^2_x\Phi$ is positive definite. It implies that $D\Phi$ is a diffeomorphism onto its image and in particular that $\dom(\Phi^*)$ is of dimension $m$. This justifies the fact that we used the interior instead of relative interior. Let $p\in \interior(\dom (\Phi^*))$ and let us write $\mu:=D\Phi$. By \cref{eq:legtr grad}, replacing $x$ by $(\mu)^{-1}p$, we have 
  \begin{equation}
    \Phi^*(p)=\langle p, \mu^{-1}(p)\rangle-\Phi\left(\mu^{-1}(p)\right).
  \end{equation}
  Differentiating in $p$, we obtain
  \begin{equation}
    D_p\Phi^*= \mu^{-1}(p)+\langle p,D_p(\mu^{-1})\cdot\rangle-D_{\mu^{-1}(p)}\Phi\circ D_p(\mu^{-1}) 
  \end{equation}
  Now it is enough to note that $D_{\mu^{-1}(p)}\Phi=\mu(\mu^{-1}(p))=p$ and we obtain $D_p\Phi^*= \mu^{-1}(p)$ which is what we wanted.
\end{proof}

We deduce the following.

\begin{corollary}\label{cor: LT Hessians}
  Under the hypotheses of \cref{prop:LT C2}, the Hessian of $\Phi$ and $\Phi^*$ are inverse of each other. In particular $\det (D_p\Phi^*)=\det(D^2_x\Phi)^{-1}$.
\end{corollary}

\subsection{Quadratic forms and Ellipsoids}\label{sec:Qform}
In this section, we fix notation for quadratic forms and prove some elementary facts about volumes of corresponding ellipsoids. Since we will later apply this to manifolds, it is convenient to state results independently of the choice of a basis and Euclidean structure. Recall that $V$ denotes a real vector space of dimension $m$.

A \emph{volume form} on $V$ is the choice of a non zero element $\omega\in \Lambda^mV^*$. A positive definite quadratic form $Q$ on $V^*$ induces a volume form $\omega:= \epsilon_1\wedge\cdots\wedge \epsilon_m$ where $\epsilon_1,\ldots,\epsilon_m$ is an orthonormal basis of $V^*$ for $Q$. It induces a notion of volume, that we denote  $\vol_Q$, given for any measurable subset $A\subset V$ by
\begin{equation}\label{eq:volQ}
  \vol_Q(A):= \int_A |\omega|
\end{equation}
where $|\omega|$ denotes the density induced by $\omega$.

\begin{proposition}
  Let $Q$ be a positive definite quadratic form on $V^*$ and let $L\in Gl(V^*)$, then $Q\circ L$ is positive definite and we have:
  \begin{equation}\label{eq:volQ'volQ}
    \vol_{Q\circ L}=\frac{1}{|\det L|}\vol_Q.
  \end{equation} 
\end{proposition}
\begin{proof}
  If $\epsilon_1,\ldots,\epsilon_m$ is a $Q$-orthonormal basis, then $\epsilon_1',\ldots,\epsilon_m'$ with $\epsilon_i':=L^{-1}\epsilon_i$ is a $(Q\circ L)$-orthonormal basis. Then we have $|\epsilon_1'\wedge\cdots\wedge\epsilon_m'|=|\det L^{-1}|\,|\epsilon_1\wedge\cdots\wedge\epsilon_m|$. Then $\det L^{-1}=1/\det L$ gives the result.
\end{proof}

Given a positive definite quadratic form $Q$ on $V$, we denote by $Q^\circ$ the dual quadratic form on $V^*$. Recall that it is given for every $u\in V^*$ by
\begin{equation}\label{eq:defQcirc}
  Q^\circ(u)=\sup\set{\langle u, v\rangle^2}{Q(v)\leq 1}.
\end{equation}

The dual quadratic form transforms under linear maps as follows.

\begin{proposition}\label{prop:QLcirc}
  For all $L\in Gl(V)$, we have
\begin{equation}\label{eq:Qlcirc}
  (Q\circ L)^\circ=Q^{\circ}\circ L^{-T}
\end{equation}
were $L^{-T}$ is the inverse of the transpose $L^T$ (recall its definition in \eqref{def:LT}).
\end{proposition}
\begin{proof}
  Directly from the definition \eqref{eq:defQcirc} we have that 
  \begin{equation}
  (Q\circ L)^\circ(u) =\sup\set{\langle u, v\rangle^2}{Q(L(v))\leq 1} = \sup\set{\langle u, L^{-1}v\rangle^2}{Q(v)\leq 1}
  \end{equation}
  and we conclude with the definition of the transpose \eqref{def:LT}.
\end{proof}

This duality satisfies $(Q^\circ)^\circ=Q$ and is order reversing in the following sense.

\begin{lemma}
  Let $Q,Q'$ be two positive definite quadratic forms on $V$ such that for all $v\in V$, we have $Q(v)\leq Q'(v)$ then for all $u\in V^*$, we have $Q^\circ(u)\geq (Q')^\circ(u)$.
\end{lemma}
\begin{proof}
  This follows directly from the definition of the dual \eqref{eq:defQcirc} and $Q'(v)\leq 1\Rightarrow Q(v)\leq 1$.
\end{proof}

For every $u, u'\in V^*$, we denote by $u^2$, respectively $u\cdot u'$, the quadratic form on $V$ given for all $v\in V$ by:
\begin{equation}\label{eq:not qf}
  u^2(v):=\langle u,v\rangle^2, \quad \text{respectively }\, (u\cdot u')(v):=\langle u, v \rangle \langle u', v \rangle.
\end{equation}

In more abstract terms, we are just identifying the space of quadratic forms on $V$ with the symmetric $2$-power of $V^*$.

We note that, given $u_1,\ldots,u_d\in V^*$, the quadratic form $\sum_{i=1}^d u_i^2$ is positive definite if and only if the family $u_1,\ldots,u_d$ spans the whole space $V^*$.

Moreover, given a positive definite quadratic form $Q$ on $V$, we denote the corresponding ellipsoid 
\begin{equation}
  B_Q:=\set{v \in V}{Q(v)\leq 1}
\end{equation}
which is the unit ball for the norm defined by $Q$.
We note that its support function is given for all $u\in V^*$ by $h_{B_Q}(u)=\sqrt{Q^\circ(u)}.$ Moreover, note that, for all $L\in Gl(V)$, we have
\begin{equation}\label{eq:BQP}
  B_{Q\circ L}=L^{-1}(B_Q).
\end{equation}

We also adopt the following notation:
\begin{equation}
  B_Q^\circ:=B_{Q^\circ}.
\end{equation}
Note that it coincides with the notion of \emph{polar body} of a convex body, see \cite[Chapter~1.6.1]{bible}. From \cref{prop:QLcirc}, we deduce that, for all linear map $L:V\to V$, the dual ball transforms as follows:
\begin{equation}\label{eq: BQL dual}
  B_{Q\circ L}^\circ=L^T(B_Q^\circ).
\end{equation}

\begin{proposition}\label{prop:volQBQ}
  Let $Q$ be a positive definite quadratic form on $V$. We have:
  \begin{equation}\label{eq:volballdual}
    \vol_Q(B_{Q^\circ})=\ball_m
  \end{equation}
  where $\ball_m=\pi^{n/2}/\Gamma(n/2+1)$ is the Euclidean volume of the standard Euclidean unit ball.
\end{proposition}
\begin{proof}
  We choose a $Q$-orthonormal basis to identify $V$ with $\bbR^m$ then $Q$ is, in this basis, the standard Euclidean quadratic form $\|\cdot\|^2$ and thus $B_Q$ and $B_{Q^\circ}$ are identified with the standard Euclidean ball of $\bbR^m$ and this gives what we want.
\end{proof}

In our study, we will have to consider ellipsoids $B_{Q^\circ}$ and $B_{(Q+u^2)^\circ}$ where $Q$ is a positive definite quadratic form of $V$ and $u\in V^*$. Next we show how their volume is related.

\begin{lemma}\label{lem:volQ+u^2}
  Let $Q$ be a positive definite quadratic form on $V$ and let $u\in V^*$. Then, for every $Q'$ positive definite quadratic form on $V$, we have 
  \begin{equation}
    \vol_{Q'}((B_{(Q+u^2)})^\circ)=\sqrt{1+Q^\circ(u)}\vol_{Q'}(B_{Q}^\circ).
  \end{equation}
\end{lemma}

\begin{proof}
  Note that, because of the rule of transformation of volume \eqref{eq:volQ'volQ}, we only need to prove the equality for one quadratic form $Q'$, and we will do it for $Q'=Q$. In that case, we choose a $Q^\circ$-orthonormal basis $\epsilon_1,\ldots,\epsilon_m$ of $V^*$ in such a way that $u$ is colinear with $\epsilon_1$, i.e $u=\sqrt{Q^\circ(u)} \epsilon_1$. Thus we have $Q+u^2=(1+Q^\circ(u))\epsilon_1^2+\epsilon_2^2+\cdots+\epsilon_m^2=Q\circ L$, where the linear map $L:V\to V$ is given, in the basis dual to $\epsilon_1,\ldots,\epsilon_m$, by the diagonal matrix:
  \begin{equation}\label{eq:pdiag}
    L:=\left(\begin{matrix}
                \sqrt{1+Q^\circ(u)}&     &         &    \\
                             & 1   &        &    \\
                             &     &  \ddots &   \\
                             &     &         & 1 
    \end{matrix} \right).
  \end{equation}
  It follows that, $\det(L)=\sqrt{1+Q^\circ(u)}$ and thus we obtain, using \eqref{eq:volballdual} and \eqref{eq:volQ'volQ} :
\begin{equation}
  \vol_{Q}(B_{Q^\circ})=b_m=\vol_{Q+u^2}(B_{(Q+u^2)^\circ})=\frac{1}{\sqrt{1+Q^\circ(u)}}\vol_{Q}(B_{(Q+u^2)^\circ})
\end{equation}
which is what we wanted.
\end{proof}

The following will also be useful.

\begin{proposition}\label{prop: proj ker}
  Let $Q$ be a positive definite quadratic form on $V$ and let $u\in V^*$. Let $L:V^*\to W^*$ be a linear map such that $u\in \ker L$. Then we have $L(B_{(Q+u^2)^\circ})=L(B_{Q^\circ})$.
\end{proposition}
\begin{proof}
  For this proof, the point of view of support functions is useful. First we note that $u\in \ker L$ if and only if the image of $L^T$ is contained in $u^\perp$. In particular, $(Q+u^2)\circ L^T=Q\circ L^T$. We obtain
  \begin{equation}
    h_{L(B_{(Q+u^2)^\circ})}=\sqrt{(Q+u^2)\circ L^T}=\sqrt{Q\circ L^T}= h_{L(B_{Q^\circ})}
  \end{equation}
  which is what we wanted.
\end{proof}

Let us now prove some bounds.

\begin{proposition}\label{prop:2-sum inclu}
  Let $Q_1,\ldots,Q_n$ be positive definite quadratic forms on $V$, let $t_1,\ldots,t_n>0$ and write $Q:=t_1 Q_1+\cdots +t_nQ_n$ and $t:=t_1+\ldots+t_n$. Then we have 
  \begin{equation}
    \frac{t_1}{\sqrt{t}}B_{Q_1}^\circ+\cdots +\frac{t_n}{\sqrt{t}}B_{Q_n}^\circ\subset B^\circ_Q\subset\sqrt{t_1}B_{Q_1}^\circ+\cdots +\sqrt{t_n}B_{Q_n}^\circ.
  \end{equation} 
\end{proposition}
\begin{proof}
We recall \cref{eq:Incl on supp}, which says that the inclusion of convex bodies is characterized by pointwise order of support functions. Let us write $h:=h_{B_Q^\circ}=\sqrt{Q}$ and $h_j:=h_{B_{Q_j}^\circ}=\sqrt{Q_j}$. Then, for every $u\in V^*$, we have:
\begin{equation}
  h(u)=\sqrt{Q(u)}=\sqrt{\sum_{j=1}^n t_j h_j(u)^2}\leq \sum_{j=1}^n \sqrt{t_j} h_j(u)
\end{equation}
which gives the upper bound. For the other inclusion, we rewrite $h(u)=\sqrt{t}\sqrt{\sum_{j=1}^n \frac{t_j}{t} h_j(u)^2}.$
Then, from the concavity of the square root, applying Jensen's inequality, we get $h(u)\geq \sqrt{t}\sum_{j=1}^n \frac{t_j}{t} h_j(u)$ which is what we wanted.
\end{proof}

We now specialize to the case $V=\bbR^m$ endowed with the standard scalar product $\langle\cdot,\cdot\rangle$ which we use to identify $\bbR^m$ with its dual. If $Q$ is a quadratic form on $\bbR^m$, there is a unique symmetric $m\times m$ matrix $M_Q$ such that for all $x\in \bbR^m$, we have $Q(x)=\langle x, M_Q x \rangle$. In that case, we write
\begin{equation}
  \det Q := \det M_Q.
\end{equation}
 
This determinant has the following geometric interpretation.

\begin{proposition}\label{prop:vol is sqdet}
  If we denote $\vol$ for the volume with respect to the Euclidean volume form, and if $Q$ is a positive definite quadratic form on $\bbR^m$, we have 
  \begin{equation}
    \vol(B_Q^\circ)=b_m \sqrt{\det Q}
  \end{equation}
\end{proposition}

\begin{proof}
  Let us denote $Q_0$ the Euclidean quadratic form. Moreover, let $L:\bbR^m\to \bbR^m$ be such that $M_Q=L^T L$. Then we have $Q=Q_0\circ L$. Using the rules of transformation \cref{eq:BQP} and \cref{eq:Qlcirc}, we get that $B_Q^\circ=L (B)$ where $B=B_{Q_0}$ the standard Euclidean unit ball. The result then follows from the fact that $\det(L)=\sqrt{\det M_Q}=\sqrt{\det Q}$.
\end{proof}

\begin{definition}\label{def:Qleqr}
  If $Q$ is a quadratic form on $\bbR^m$ and $\lambda\in \bbR$, then we write
  \begin{equation}
    Q\leq \lambda \quad \iff \quad Q(x)\leq \lambda \|x\|^2,\, \forall x\in \bbR^m\setminus \{0\}
  \end{equation}
  and similarly for $>,\leq$ and $<$. In particular $Q>0$ means that $Q$ is positive definite.
\end{definition}

\begin{proposition}
  Let $Q$ be a quadratic form on $\bbR^m$, let $\lambda\in\bbR$ and suppose that $Q\leq \lambda$. Then $\det Q\leq \lambda^m$.
\end{proposition}
\begin{proof}
  Since we have $\sqrt{Q(x)}\leq \sqrt{\lambda} \|x\|$, it follows that $B_Q^\circ\subset \sqrt{\lambda} B$ where $B=B^\circ$ is the standard Euclidean ball. This implies that $\vol(B_Q^\circ)\leq \lambda^{\frac{m}{2}} b_m$. The result then follows by \cref{prop:vol is sqdet}.
\end{proof}

By simply applying Cauchy-Schwartz, we have the following.
\begin{proposition}\label{prop:QCS}
Let $u_1,\ldots, u_k\in \bbR^m$ and let $Q:=\sum_{i=1}^k u_i^2$. Then we have $Q\leq \sum_{i=1}^k\|u_i\|^2$.
\end{proposition}

\begin{proposition}\label{prop:Duality lambda 1/lambda}
  Let $Q>0$ be a positive definite quadratic form on $\bbR^m$, let $\lambda>0$ and suppose that $Q\leq \lambda$. Then we have $Q^\circ\geq \lambda^{-1}$.
\end{proposition}

\begin{proof}
  This follows directly from the fact that the matrix of $Q^\circ$ is given by $M_{Q^\circ}=M_Q^{-1}$. Thus the eigenvalues of $M_Q$ are bounded from above by $\lambda$ if and only if the eigenvalues of $M_{Q^\circ}$ are bounded from below by $1/\lambda$.
\end{proof}

We conclude this section by proving a convenient relation between volumes of balls and spheres.

\begin{lemma}\label{lem:de l'orangeraie}
  Recall that we denote by $b_m$, respectively $s_m$ the $m$-dimensional volume of the Euclidean unit ball $B\subset\bbR^m$, respectively unit sphere $S^m\subset \bbR^{m+1}$. We have:
  \begin{equation}
    \frac{m! b_m s_m}{(2\pi)^m}=2.
  \end{equation}
\end{lemma}

\begin{proof}
  We use the explicit formula for $b_m$ and $s_m$ in terms of the Gamma function, see \cite[\href{https://dlmf.nist.gov/5.19}{5.19(iii)}]{NIST:DLMF}:
  \begin{equation}
    b_m=\frac{2\pi^{\frac{m}{2}}}{m\Gamma\left(\frac{m}{2}\right)} \mathand s_m=\frac{2\pi^{\frac{m+1}{2}}}{\Gamma\left(\frac{m+1}{2}\right)}.
  \end{equation}
  Legendre's duplication formula, see \cite[\href{https://dlmf.nist.gov/5.5}{5.5(iii)}]{NIST:DLMF} states that $\Gamma\left(\frac{m}{2}\right)\Gamma\left(\frac{m+1}{2}\right)=2^{1-m}(m-1)!\sqrt{\pi}.$ This yields:
  \begin{equation}
    b_ms_m=\frac{2^2\pi^{m+\frac{1}{2}}}{m}\frac{1}{2^{1-m}(m-1)!\sqrt{\pi}}=\frac{2(2\pi)^m}{m!}
  \end{equation}
  which is what we wanted.
\end{proof}

\subsection{Gaussian vectors}

A random vector $X\randin V$ is called \emph{Gaussian} if for every $u\in V^*$, the random variable $\langle u, X\rangle \randin \bbR$ is Gaussian. The law of a Gaussian random vector is characterized by its \emph{mean} $\bbE X\in V$ and covariance $\Sigma$ which is the Quadratic form on $V^*$ given for all $u\in V^*$ by $\Sigma(u):=\bbE \left[ \langle u, X-\bbE X \rangle^2 \right].$
A Gaussian vector is said to be \emph{non degenerate} if its covariance is non degenerate, that is if it is positive definite. In what follows, we will only consider \emph{centered} Gaussian vectors, i.e. such that $\bbE X=0$. In that case, we note that the covariance is given for all $u\in V^*$ by 
\begin{equation}
  \Sigma(u):=\bbE \left[ \langle u, X \rangle^2 \right].
\end{equation}

If $V$ is Euclidean, we call \emph{standard Gaussian vector} the centered Gaussian vector whose covariance coincides with the quadratic form induced by the scalar product.

First we compute the Vitale zonoid of a Gaussian vector. Recall that the support function of the Vitale zonoid is given by \cref{eq:sptfct Vitale}.

\begin{proposition}\label{prop:EX=B}
  Let $X\randin V$ be a non degenerate centered Gaussian vector with covariance $\Sigma$. Then we have $\bbE \seg{X}=\frac{1}{\sqrt{2\pi}}B_{\Sigma}^\circ.$
\end{proposition}

\begin{proof}
  First we note that, if $L\in Gl(V)$ then $L(X)$ is a centered Gaussian vector with covariance $\Sigma\circ L^T$. Thus by \cref{eq:BQP} and \cref{eq:Vitale linear}, it is enough to prove it for one (centered non degenerate) Gaussian vector, thus we will assume that $X\randin\bbR^m$ is a standard Gaussian vector. In that case, $\bbE \seg{X}$ is invariant under the action of the orthogonal group $O(m)$ and hence must be a (standard Euclidean) ball. To compute its radius, one note that, if $e_1,\ldots,e_m$ is the standard basis of $\bbR^m$, then $\langle e_1, X\rangle\randin \bbR$ is a standard Gaussian variable and thus $\bbE|\langle e_1, X\rangle|=\sqrt{\tfrac{2}{\pi}}$. Dividing by $2$ the gives the result.
\end{proof}

\subsection {The Kac-Rice formula with zonoids}\label{sec:KR}
The Kac-Rice formula is a tool developed to estimate the size of the zero set of certain random fields. The idea dates back to Kac \cite{Kac} who used it to study the average number of zeros of a random polynomial. 
In this section, we partially introduce the framework developed by the author and Michele Stecconi in \cite{ZonSec} which allows to interpret the Kac-Rive formula in terms of operations on certain zonoids.
Before doing so, let us give, informally, the idea behind the Kac-Rice formula for the reader that never encountered it before.

Suppose that we have a function $f:\bbR\to\bbR$ for which $0$ is a regular value, i.e., for all $x\in \bbR$ such that $f(x)=0$, $f'(x)\neq 0$. Then we can see that for every interval $(a,b)\subset\bbR$
\begin{equation}
  \#\sol{f;(a,b)}=\lim_{\varepsilon\to 0}\frac{1}{2\varepsilon}\int_a^b |f'(x)| \, \mathds{1}_{|f(x)|<\varepsilon}\,\mrd x
\end{equation}
where $\mathds{1}_{|f(x)|<\varepsilon}$ is the characteristic function of the set $\set{x\in\bbR}{|f(x)|<\varepsilon}$. 
Now, for a \emph{suitable} random function $\GRF:\bbR\to\bbR$, one could exchange limits and integrals to obtain
\begin{equation}
  \bbE\#\sol{\GRF;(a,b)}=\int_a^b\lim_{\varepsilon\to 0}\frac{1}{2\varepsilon} \mathbb{P}\left(|\GRF(x)|<\varepsilon\right)\,\bbE\left[|\GRF'(x)|\, |\, |\GRF(x)|<\varepsilon\right]\,\mrd x.
\end{equation}
Then, for a fixed $x\in (a,b)$, we have, again under suitable assumption on $\GRF$, $\frac{1}{2\varepsilon} \mathbb{P}\left(|\GRF(x)|<\varepsilon\right)\to \rho_{\GRF(x)}(0)$ the density of the random variable $\GRF(x)\randin \bbR$ in $0$ and the random function $(\GRF\, |\, |\GRF(x)|<\varepsilon)$ should converge to a random field $(\GRF\, |\,\GRF(x)=0)$. 
The whole subtlety of the Kac Rice formula is to determine for which class of random fields such statements are valid. Those are in general deep results. In that case, we obtain the most classic form of the Kac Rice formula:
\begin{equation}\label{eq:KR classic}
  \bbE\#\sol{\GRF;(a,b)}=\int_a^b\rho_{\GRF(x)}(0)\,\bbE\left[|\GRF'(x)|\, |\, \GRF(x)=0\right]\,\mrd x.
\end{equation}
One of the reason why this formula is so useful is that it belongs to this family of formulas that allows to compute a \emph{global} quantity (in that case the number of zeros) as the integral of a \emph{local quantity}. This formula can then be generalized to higher dimensions and to manifolds. The reader can refer to \cite{AdlerTaylor,AzaisWschebor}.

In \cite{ZonSec}, Stecconi and the author study appropriate random fields that they call ZKROK fields to which a Kac-Rice formula similar to \cref{eq:KR classic} in higher dimension and on manifolds apply. To see the precise definition of a ZKROK field, see \cite[Definition~4.1]{ZonSec}. Here, we choose to present the case of smooth random fields that are supported on a finite dimensional space of functions. This has the advantage of being less technical than the full ZKROK hypotheses but still give a grasp of the generality of the theory and being suitable for a lot of application, including ours. This is contained in \cite[Section~10.3]{ZonSec}.

Moreover, the work of \cite{ZonSec} is to interpret the Kac Rice density (that is the integrand in \cref{eq:KR classic}) in terms of geometric quantities of certain \emph{zonoids}. This might seem like an extra step in the computation of the density, however, we believe that the advantage is twofold. First it allows to compute a ``density'' even for fields of non zero codimension. It is the \emph{zonoid section} below that plays this role. Moreover it gives a very geometrical picture of this density. This second point becomes particularly valid when studying monotonicity, see \cref{fig:ellip} below. It is then a matter of comparing zonoids or, in our case, ellipsoids.

\begin{definition}\label{def:admissible}
  A random field $\GRF\randin C^\infty(m,\bbR^k)$ is said to be \emph{admissible} if there is a finite dimensional subspace $\scrF\subset C^\infty(M,\bbR^k)$ endowed with a scalar product such that $\GRF\in \scrF$ almost surely and such that the following holds: 
  \begin{enumerate}[label=(\roman*)]
    \item for all $x\in M$, the linear map $ev_x:\scrF\to\bbR^k$ given for all $f\in \scrF$ by $ev_x(f):=f(x)$ is surjective; \label{def:admissible -F}
    \item  $\GRF$ admits a continuous density $\rho_\GRF:\scrF\to \bbR_+$ such that $\rho_\GRF(0)>0$ and  there exists $\alpha>\dim \scrF$ such that $\rho_\GRF(f)=O(\|f\|^{-\alpha})$ when $\|f\|\to\infty$. \label{def:admissible -rho}
  \end{enumerate} 
  When $k=1$, we say that the random field is \emph{scalar}. 
\end{definition}

Note that being admissible implies that $\GRF^{-1}(0)$ is a smooth submanifold of codimension $k$, see the proof in \cite[Proposition~10.7]{ZonSec}. Moreover, one can see that, because $\scrF$ is finite dimensional, being admissible does not depend on the choice of the scalar product on $\scrF$. Moreover, as noted in \cite{ZonSec} positivity in $0$ condition is probably not needed.

If $\scrF\subset C^\infty(M,\bbR^k)$ is a finite dimensional subspace satisfying \cref{def:admissible -F}, then for all $x\in M$, we consider the subspace
\begin{equation}
  \scrF_x:=\set{f\in\scrF}{f(x)=0}\subset\scrF.
\end{equation}
Equivalentely, one has $\scrF_x=\ker ev_x$. In particular, because of the surjectivity condition, this subspace has codimension $k$. Moreover, note that in the case $k=1$, $\scrF_x=ev_x^\perp$.

We then build, for every $x\in M$, a zonoid in the cotangent space $\Lambda^kT^*_x M$ associated to admissible random fields. Recall the notation for segments \cref{eq:seg} and the definition of the Minkowski integral in \cref{def:Minkowski integral}.

\begin{definition}\label{def:zon sec}
  For every admissible random field $\GRF\randin \scrF\subset C^\infty(M,\bbR^k)$ and every $x\in M$, we let:
  \begin{equation}
    \zeta_\GRF(x):=\frac{1}{\|ev_x^1\wedge\cdots\wedge ev_x^k\|}\int_{\scrF_x} \quad \seg{D_xf^1\wedge \cdots\wedge D_xf^k}\quad \rho_\GRF(f) \,\mrd f
  \end{equation}
  where $ev_x:=(ev_x^1,\ldots,ev_x^k)\in (\scrF^*)^k$ is the evaluation at a point, i.e. $ev_x(f):=f(x)$ for all $f \in \scrF$, $D_xf^i$ denotes the differential of the function $f^i$ at the point $x\in M$ and where $\mrd f$ denotes the integration with respect to the Lebesgue measure on the subspace $\scrF_x$.
\end{definition}

One of the main results of \cite{ZonSec} states that these zonoids compute geometric quantities of the zero set. We fix a Riemannian metric on $M$. The following is \cite[Corollary 7.2.]{ZonSec}.

\begin{proposition}\label{prop:mainZKROK}
  Let $\GRF_1\randin C^\infty\left(M,\bbR^{k_1}\right),\ldots,\GRF_n\randin C^\infty\left(M,\bbR^{k_n}\right)$ be independent admissible random fields such that $k_1+\cdots+k_n=m$ and let $U\subset M$ be an open set, we have:
  \begin{equation}
    \bbE\#\sol{\GRF_1,\ldots,\GRF_n;U}=\int_U\ell(\zeta_{\GRF_1}(x)\wedge\cdots\wedge \zeta_{\GRF_n}(x)) \mrd M(x).
  \end{equation}
  In the case where $n=m$ and $k_1=\cdots,k_n=1$, we get
  \begin{equation}
    \bbE\#\sol{\GRF_1,\ldots,\GRF_n;U}=m!\int_U\MV(\zeta_{\GRF_1}(x),\ldots, \zeta_{\GRF_m}(x)) \mrd M(x),
  \end{equation}
  and in particular if $\GRF_1,\ldots,\GRF_m$ are identically distributed:
  \begin{equation}
    \esol{\GRF_1;U}=m!\int_U\vol(\zeta_{\GRF_1}(x)) \mrd M(x)
  \end{equation}
  where the volume and mixed volume are computed in the metric given by the Riemannian structure.
\end{proposition}

The zonoid section also satisfies a nice pull-back property. The following is \cite[Theorem~B]{ZonSec} in the particular case of an embedding.

\begin{proposition}\label{prop:pull back}
  Let  $\iota:S\hookrightarrow M$ be an embedded submanifold. Then $\GRF|_{S}:= \GRF\circ \iota$ is admissible and we have for all $x\in S$:
  \begin{equation}
    \zeta_{\GRF|_{S}}(x)=\pi_x(\zeta_{\GRF}(x))
  \end{equation}
  where $\pi_x:= (D_x\iota)^T:T^*_xM\to T^*_x S$ is the projection induced by the embedding.
\end{proposition}

\begin{remark}
  The transversality assumption in \cite[Theorem B]{ZonSec} is not needed in the particular case of admissible fields. Indeed if $\GRF$ is admissible, it is clear that $\GRF|_{S}$ is also admissible. Therefore $0$ is a regular value of $\GRF|_{S}$ almost surely which is equivalent to saying that $\GRF^{-1}(0)$ and $S$ are transversal almost surely.
\end{remark}

\section{GRFs and Adler-Taylor theory}\label{sec:GRFs and AT}
We now specialize to the Gaussian case. We fix an admissible scalar (centered) Gaussian Random Field (GRF), that is, a non degenerate Gaussian vector $\GRF\randin \scrF\subset C^\infty(M)$ in a subspace $\scrF$ of dimension $d<\infty$ that satisfies the surjectivity condition $(i)$ in \cref{def:admissible}. 

\subsection{General theory}
We start with the following definition.
\begin{definition}
  The \emph{covariance function} of $\GRF$ is the function $\cov:M\times M\to \bbR$ given by 
  \begin{equation}
    \cov(x,y):=\bbE \left[\GRF(x)\GRF(y)\right].
  \end{equation}
  We will also consider the \emph{diagonal covariance function} $\cov(x):=\cov(x,x)$. We will often abuse notation and call the latest also the covariance function.
\end{definition}

Note that, if $\GRF$ is admissible, the surjectivity condition in \cref{def:admissible} is equivalent to the fact that for all $x\in M$, there is $f\in\scrF$ such that $f(x)\neq 0$. It implies that $\cov(x)$ is strictly positive for all $x\in M$. The covariance \emph{function} $\cov$ is not to be confused with the covariance $\Sigma$ of $\GRF$ as a Gaussian vector of $\scrF$ which is a quadratic form on $\scrF^*$. Although, one can check that the covariance function can be obtained by evaluating $\Sigma$ at the linear form $ev_x\in \scrF^*$.

In the following we endow $\scrF$ with the scalar product induced by the covariance $\Sigma$ of $\GRF$.

\begin{remark}\label{rk:ev is K}
  If one uses this scalar product to identify $\scrF^*\cong\scrF$, then, $ev_x\in \scrF\subset C^\infty(M)$ and one can then check that, as a function, it is given for all $y\in M$ by $ev_x(y)=\cov(x,y)$. Indeed, by definition of the scalar product, we have $ev_x(y)=\langle ev_y, ev_x\rangle=\bbE \left[ \GRF(y)\GRF(x)\right]=\cov(x,y)$. The covariance function $\cov(x,y)$ is alo called a \emph{reproducing kernel} and $\scrF$ a (real) Reproducing Kernel Hilbert Space (RKHS). A whole theory of RKHS was developed by Aronszajn in the 40s and 50s, see \cite{AronszajnTRK}. 
\end{remark}

\begin{lemma}\label{lem:K=ev}
  For all $x\in M$, we have $\cov(x)=\|ev_x\|^2$.
\end{lemma}
\begin{proof}
  It is just a matter of applying the definitions, $\|ev_x\|^2=\Sigma(ev_x)=\bbE[\langle ev_x,\GRF\rangle^2]=\bbE[\GRF(x)^2]=\cov(x).$
\end{proof}

In the case where the random field is Gaussian, the zonoid section from \cref{def:zon sec} is an \emph{ellipsoid} section that we can compute explicitly. To do that, we consider, for all $x\in M$, the linear map 
\begin{equation}
  D_x:\scrF\to T^*_xM
\end{equation}
such that $D_xf$ is the differential of the function $f$ at the point $x$.

\begin{proposition}\label{prop:zeta as ellips}
  Let $\GRF\randin \scrF\subset C^\infty(M)$ be an admissible centered Gaussian random field with covariance $\Sigma$. Endow $\scrF$ with the scalar product induced by $\Sigma$. Then, for all $x\in M$, we have
  \begin{equation}
    \zeta_\GRF(x)=\frac{1}{2\pi \sqrt{\cov(x)}}(D_x\circ \pi_x)(B_{\Sigma}^\circ)
  \end{equation}
  where $\pi_x:\scrF\to \scrF_x$ is the orthogonal projection and where recall that $B_\Sigma^\circ$ denotes the dual of the unit ball of the quadratic form $\Sigma$.
\end{proposition}

\begin{proof}
  First we note that, identifying $\scrF^*\cong\scrF$ with the scalar product, the transpose of $\pi_x$ is the inclusion $\pi_x^T:\scrF_x\hookrightarrow \scrF.$ We have, by \cref{prop:EX=B}:
  \begin{equation}\label{eq:BSigma=intFx}
    \frac{1}{\sqrt{2\pi}}\pi_x(B_{\Sigma}^\circ)=\bbE\seg{\pi_x(\GRF)}=\int_\scrF \seg{\pi_x(f)} \quad \frac{e^{-\frac{\|f\|^2}{2}}}{(2\pi)^{\frac{d}{2}}} \mrd f=\int_{\scrF_x} \seg{f} \quad \frac{e^{-\frac{\|f\|^2}{2}}}{(2\pi)^{\frac{d-1}{2}}} \mrd f
  \end{equation}
  where in the last equality we integrated over the fibers of the orthogonal projection $\pi_x$. Note that we are integrating \emph{segments} in the sense of \cref{def:Minkowski integral}, this can be unsettling and the skeptical reader can always replace them with their support function evaluated at a point. However, the author finds it more elegant and direct this way and hopes to convince at least some of the readers.

  It remains only to note that $\rho_\GRF(f)=\frac{e^{-\frac{\|f\|^2}{2}}}{(2\pi)^{\frac{d}{2}}}= \frac{1}{\sqrt{2\pi}}\frac{e^{-\frac{\|f\|^2}{2}}}{(2\pi)^{\frac{d-1}{2}}}.$
  Reintroducing in \cref{eq:BSigma=intFx} and applying the linear map $D_x$ gives 
  \begin{equation}
  \frac{1}{2\pi}(D_x\circ\pi_x)(B_{\Sigma}^\circ)=\int_{\scrF_x} \seg{D_xf} \quad \rho_\GRF(f) \mrd f.
  \end{equation}
  The result follows by dividing both side by $\sqrt{\cov(x)}$ which, by \cref{lem:K=ev}, is equal to $\|ev_x\|$.
\end{proof}

In the spirit of Adler and Taylor \cite{AdlerTaylor}, we want to consider this ellipsoid as the unit ball of a certain Riemannian metric.

\begin{definition}\label{def:AT metric}
  We say that the Gaussian random field $\GRF$ is \emph{non degenerate} if $\zeta_\GRF(x)$ is full dimensional for all $x\in M$. In that case, we define the \emph{Adler-Taylor} (AT) metric on $M$ to be the Riemannian metric such that for all $x\in M$, its unit ball $B_x\subset T_xM$ satisfy 
  \begin{equation}
    \zeta_\GRF(x)=\frac{1}{2\pi}B_x^\circ.
  \end{equation}
\end{definition}

\begin{remark}
Note that $\GRF$ being a non degenerate random field is different than $\GRF\randin \scrF$ being non degenerate as a Gaussian vector in $\scrF$. In fact, as mentioned at the beginning of this section, we always assume it is non degenerate \emph{as a Gaussian vector} which is no restriction since we can always restrict to the support of $\GRF$ in $C^\infty(M)$. To see that this does not imply being non degenerate \emph{as a GRF} one can consider the case where $\scrF$ is one dimensional and spanned by the constant function in which case, everything is trivially mapped to zero. 
\end{remark}

The AT metric then compute the expected number of solution of the random field.

\begin{proposition}\label{prop:esolisATvol}
  Let $\GRF$ be a non degenerate admissible scalar GRF and let $U\subset M$ be an open set. If we endow $M$ with the AT metric induced by $\GRF$, we obtain
  \begin{equation}
    \esol{\GRF;U}=\frac{2}{s_m}\int_U \mrd M
  \end{equation}
  where $\mrd M$ denotes the integration with respect to the Riemannian volume form, i.e., $\int_U \mrd M$ is the Riemannian volume of $U$, and where recall that $s_m$ is the volume of the Euclidean unit sphere $S^m\subset\bbR^{m+1}$. 
\end{proposition}

\begin{proof}
  By definition of the AT metric and by \cref{prop:volQBQ}, we have for all $x\in M$ $\vol_m(\zeta_\GRF(x))=\frac{\ball_m}{(2\pi)^{m}}.$
  Then, by \cref{prop:mainZKROK}, we have $\esol{\GRF;U}=\frac{m! \ball_m}{(2\pi)^{m}}\int_U \mrd M.$ 
  We conclude using the relation between volumes of spheres and balls in \cref{lem:de l'orangeraie}.
\end{proof}

In order to express the AT metric we will use the following notions that will become central in our later study of exponential sums.

\begin{definition}\label{def:momentmap}
 We define the following \emph{potential} function $\Phi:M\to\bbR$ given for all $x\in M$ by 
 \begin{equation}
  \Phi(x):=\frac{1}{2}\log (\cov(x)).
 \end{equation}
 Moreover, the \emph{moment map} is defined to be the one form $\mu:M\to T^*M$ given for all $x\in M$ by 
 \begin{equation}
  \mu(x):=D_x\Phi=\frac{1}{\cov(x)}\bbE\left[\GRF(x)D_x \GRF\right].
 \end{equation}
\end{definition}

The name \emph{moment map} will be justified when we will study exponential sums where it plays the same role and share many similarities with the moment map of complex and toric geometry.

\begin{remark}
It is common, when studying GRF, to reduce to the case where $\cov\equiv 1$ which is always possible by substituting $\GRF$ by $\GRF/\sqrt{\cov}$. Since we are interested in the zero set and $\cov$ is strictly positive, it does not affect the result. This makes some formula like the quadratic form $g_x$ in \cref{eq:def gx} below simpler. However the price to pay is that the functions in $\scrF$ are more complicated. Thus in explicit computations, the gain is not quite substantial. Moreover the moment map $\mu$ becomes trivial and its role his completely hidden. Hence why we chose not to use this convention and adopt coordinates and metrics that seem more suited.
\end{remark}

For computations, it can be convenient to fix an othonormal basis $f_1,\ldots,f_d\in \scrF$. In this basis the random field $\GRF$ is distributed as $\sum_i\xi_if_i$ where $\xi_i\randin \bbR$ are iid standard Gaussian Variables. In that case, we get for all $x\in M$:
\begin{equation}
  \cov(x)=\sum_{i=1}^d f_i(x)^2\quad \text{and} \quad  \mu(x)=\sum_{i=1}^d \frac{f_i(x)}{\cov(x)} D_x f_i.
\end{equation}

\begin{proposition}\label{prop:ATmetric}
  If $\GRF$ is non degenerate, the AT metric is given for every $x\in M$, $v\in T_xM$ by 
  \begin{equation}\label{eq:def gx}
    g_x(v)=\frac{1}{\cov(x)}\,\bbE\left[\left(D_x\GRF(v)-\GRF(x)\langle\mu(x),v\rangle\right)^2\right]=\frac{1}{\cov(x)}\bbE\left[\left(D_x\GRF(v)\right)^2\right]-\langle\mu(x),v\rangle^2.
  \end{equation}
\end{proposition}

\begin{proof}
  First let us prove that the two quadratic forms on the right hand side are equal. We have
  \begin{equation}
    \bbE\left[\left(D_x\GRF(v)-\GRF(x)\langle\mu(x),v\rangle\right)^2\right]= \bbE\left[D_x\GRF(v)^2\right]+\bbE\left[\GRF(x)^2\right]\langle\mu(x),v\rangle^2-2\bbE\left[\GRF(x)D_x \GRF(v)\right]\langle\mu(x),v\rangle.
  \end{equation}
  Then it is enough to note that $\bbE\left[\GRF(x)^2\right]=\cov(x)$ and $\bbE\left[\GRF(x)D_x \GRF(v)\right]=\cov(x)\langle\mu(x),v\rangle$.

  Now we let $B_x\subset T_xM$ be the unit ball for $g_x$. By definition of the AT metric and by \cref{prop:zeta as ellips}, we need to prove that $\frac{1}{ \sqrt{\cov(x)}}(D_x\circ \pi_x)(B_{\Sigma}^\circ)= B_x^\circ$. In order to do this, let us rewrite more explicitly the projection $\pi_x$. Since it is the projection onto the hyperplane orthogonal to $ev_x$ and using the scalar product to consider $ev_x$ as an element of $\scrF$, we get for all $f\in \scrF$:
  \begin{equation}
    \pi_x(f)=f-\langle ev_x,f\rangle \frac{1}{\|ev_x\|^2}ev_x=f-f(x)\frac{1}{\cov(x)}ev_x.
  \end{equation}
  Recalling \cref{prop:EX=B}, we get the following. 
  \begin{equation}\label{eq:this eq*}
    \frac{1}{\sqrt{2\pi}}(D_x\circ \pi_x)(B_\Sigma)=\bbE\seg{D_x \pi_x(\GRF)}=\bbE\seg{D_x\GRF-(\GRF(x)/\cov(x))D_xev_x}.
  \end{equation}
  Note that, for $x$ fixed, $(\GRF(x)/\cov(x))$ is a constant and not a function and thus is not ``differentiated'' by $D_x$ that we consider as a linear map $\scrF\to T^*_xM$, thus we indeed have $D_x((\GRF(x)/\cov(x))ev_x)=(\GRF(x)/\cov(x))D_xev_x$. 

  Finally, as an element of $\scrF$, we have (see \cref{rk:ev is K}):
  \begin{equation}\label{eq:ev=K}
    ev_x=\bbE[\GRF(x)\GRF]=\cov(x,\cdot).
  \end{equation}
  
  Thus we have $D_x ev_x=\bbE[\GRF(x)D_x\GRF]=\cov(x)\mu(x).$ Reintroducing in \cref{eq:this eq*}, we get 
  \begin{equation}
    \frac{1}{\sqrt{2\pi \cov(x)}}(D_x\circ \pi_x)(B_\Sigma)=\frac{1}{\sqrt{\cov(x)}}\bbE\seg{D_x\GRF-\GRF(x)\mu(x)}=\frac{1}{\sqrt{2\pi}}B_x
  \end{equation}
  where in the last equality we used \cref{prop:EX=B} with the fact that $\frac{1}{\sqrt{\cov(x)}}(D_x\GRF-\GRF(x)\mu(x))$ is a Gaussian vector with covariance $g_x$. This concludes the proof.
\end{proof}

\begin{definition}
  If $\GRF$ is non degenerate, for all $x\in M$, we denote by $g^x$ the quadratic form on $T^*_xM$ that is dual to $g_x$, i.e. in the notation of \cref{sec:Qform}, $g^x:=(g_x)^\circ$.
\end{definition}

Once again, it can be useful, for explicit computations, to express it in an orthonormal basis $f_1,\ldots,f_d\in \scrF$. Recalling the notation \eqref{eq:not qf} for quadratic forms, we obtain:
\begin{equation}\label{eq:g_x in basis}
  g_x=\frac{1}{\cov(x)}\sum_{i=1}^d\left(D_xf_i-f_i(x)\mu(x)\right)^2=\frac{1}{\cov(x)}\sum_{i=1}^d\left(D_x f_i\right)^2-\mu(x)^2
\end{equation}

It is however not clear how to express the dual $g^x$ in general. Finding upper and lower bound will be an important part of the study in the next sections.

We gather in the next proposition the equivalent formulation of the non degeneracy of $\GRF$. The proof is immediate and thus omitted.

\begin{proposition}\label{prop:non deg equiv}
  The Gaussian random field $\GRF$ is non degenerate, if and only if, one of the following properties is satisfied for every $x\in M$:
  \begin{enumerate}[label=(\roman*)]
    \item the quadratic form  $g_x$ defined by \cref{eq:def gx} is positive definite;
    \item the ellipsoid $\zeta_\GRF(x)$ is of dimension $m$;
    \item $\vol_m(\zeta_\GRF(x))>0$.
  \end{enumerate}
\end{proposition}

\subsection{Local monotonicity of GRFs}\label{sec:locmonoGRF}
We now choose $f_0\in C^\infty (M)\setminus\scrF$ and build the (admissible) random field
\begin{equation}
  \GRF_0:=\GRF+\xi_0 f_0
\end{equation}
where $\xi_0\randin \bbR$ is a standard Gaussian variable independent of $\GRF$. 
We write 
\begin{equation}
  \scrF_0:=\scrF\oplus \bbR f_0
\end{equation} 
that we endow with the scalar product that makes this splitting orthogonal and such that $\|f_0\|=1$. We also write $\cov_0$ for the (diagonal) covariance function of $\GRF_0$ which is given, for all $x\in M$ by:
\begin{equation}
  \cov_0(x)=\cov(x)+f_0(x)^2.
\end{equation}

We want to compare the number of expected solutions of $\GRF_0$ and of $\GRF$. In order to do so, we make the following definition.

\begin{definition}\label{def:tau Psi}
  We define the one form $\tau$ on $M$ given for all $x\in M$ by 
  \begin{equation}
    \tau(x):=\frac{1}{\sqrt{\cov_0(x)}}\left(f_0(x)\mu(x)-D_xf_0\right).
  \end{equation}
  Moreover, we define $\Psi\in C^\infty(M)$ given for all $x\in M$ by: 
  \begin{equation}
    \Psi(x):=\left(\frac{\cov(x)}{\cov_0(x)}\right)^{\frac{m}{2}}\sqrt{1+ g^x(\tau(x))}.
  \end{equation}
\end{definition}

The main result of this section is the following.

\begin{theorem}\label{thm:Psi}
  For every open set $U\subset M$ and any Riemannian metric on $M$, we have 
  \begin{equation}
    \esol{\GRF_0;U}-\esol{\GRF;U}=\int_U (\Psi(x)-1) \,m! \vol_m(\zeta_\GRF(x))\, \mrd M (x).
  \end{equation}
\end{theorem}

The proof is essentially based on the following lemma.

\begin{lemma}\label{lem:g0x}
  The AT metric for $\GRF_0$ is given, as a quadratic form, for all $x\in M$, by
  \begin{equation}
    (g_0)_x=\frac{\cov(x)}{\cov_0(x)}\left(g_x+\tau(x)^2\right).
  \end{equation}
\end{lemma}

\begin{proof}
  Let us choose an orthonormal basis $f_1,\ldots,f_d$ of $\scrF$. Then, $f_0,f_1,\ldots,f_d$ is an orthonormal basis of $\scrF_0$ and by \cref{eq:g_x in basis}, we first express the new moment map $\mu_0$ as follows:
  \begin{equation}\label{eq:mu0}
    \mu_0(x)=\mu(x)-\frac{f_0(x)}{\cov_0(x)}\left(f_0(x)\mu(x)-D_x f_0\right).
  \end{equation}
  Note that we then have:
  \begin{equation}\label{eq:mu0diff}
    f_0(x)\mu_0(x)-D_xf_0=\frac{\cov(x)}{\cov_0(x)}\left(f_0(x)\mu(x)-D_xf_0\right).
  \end{equation}
  We can then compute the new metric $g_0$. We omit the dependence in $x$ in the writing.
  \begin{align}
    g_0   &=\frac{1}{\cov_0}\sum_{i=0}^d(f_i\mu_0-Df_i)^2\\
              &=\frac{1}{\cov_0}\sum_{i=1}^d\left(f_i\mu-Df_i-\frac{f_if_0}{\cov_0}\left(f_0\mu-D f_0\right)\right)^2+\frac{1}{\cov_0}(f_0\mu_0-Df_0)^2\\
              &=\frac{\cov}{\cov_0}g-\frac{2 f_0}{\cov_0^2} \left(\sum_{i=1}^d f_i^2\mu-f_iDf_i\right)+ \frac{f_0^2}{\cov_0^3}\left(\sum_{i=1}^d f_i^2\right)\left(f_0\mu-D f_0\right)^2+\frac{\cov^2}{\cov_0^3}(f_0\mu-Df_0)^2
    \end{align}
    where in the second equality we used \cref{eq:mu0} and in the third we expanded the square and used \cref{eq:mu0diff}. Now we note that $\sum_{i=1}^d f_i^2\mu-f_iDf_i=K\mu-K\mu=0$. This yields
    \begin{equation}
      g_0=\frac{\cov}{\cov_0}\left(g+\frac{f_0^2+\cov}{\cov_0^2}\left(f_0\mu-D f_0\right)^2\right)=\frac{\cov}{\cov_0}\left(g+\tau^2\right)
    \end{equation}
  which is what we wanted.
\end{proof}

\begin{remark}\label{rk:gox bigger and smaller}
  We note that the factor $\cov/\cov_0$ is smaller than one (strictly where $f_0$ is non zero) and thus contributes to make the metric $g_0$ \emph{smaller} while the extra term $\tau^2$ contributes to make it \emph{bigger}.
\end{remark}

\begin{proof}[Proof of \cref{thm:Psi}]
  We already know from \cref{prop:mainZKROK} that 
  \begin{equation}
  \esol{\GRF;U}=\int_U  \,m! \vol_m(\zeta_\GRF(x))\, \mrd M (x)\quad \text{and} \quad \esol{\GRF_0;U}=\int_U  \,m! \vol_m(\zeta_{\GRF_0}(x))\, \mrd M (x).
  \end{equation}
  We will prove that, for all $x\in M$, we have:
  \begin{equation}\label{eq:volY=PsivolY}
    \vol_m(\zeta_{\GRF_0}(x))=\Psi(x)\vol_m(\zeta_{\GRF}(x))
  \end{equation}
  which will be enough. Let $B_0$ be the unit ball for the quadratic form $(g_0)_x$, recall that, by definition of the AT metric, $\zeta_{\GRF_0}(x)=(1/(2\pi))B_0^\circ$. Moreover, by \cref{lem:g0x}, we have
  \begin{equation}\label{eq:Bo as g}
    B_0^\circ=\sqrt{\frac{\cov(x)}{\cov_0(x)}}\left(B_{g_x+\tau(x)^2}\right)^\circ.
  \end{equation} 
  Which yields,
  \begin{equation}
    \vol_m(B_0^\circ)  = \left(\frac{\cov(x)}{\cov_0(x)}\right)^{\frac{m}{2}}\vol_m(B_{g_x+\tau(x)^2}^\circ) = \left(\frac{\cov(x)}{\cov_0(x)}\right)^{\frac{m}{2}}\sqrt{1+g^x(\tau(x))}\, \vol_m(B_{x}^\circ)
  \end{equation}
  where in the second equality we used \cref{lem:volQ+u^2}.

\end{proof}

The use of a basis in the proof of \cref{lem:g0x} makes it rather simple and short but maybe some deeper understanding is lost. In fact, the  decomposition of the quadratic form $(g_0)_x$ in \cref{lem:g0x}, can be seen as coming from a splitting of the subspace $(\scrF_0)_x$ which we show in the next proposition.

\begin{proposition}\label{prop:split Fox}
  For all $x\in M$, the space $\scrF_x\subset \scrF\subset \scrF_0$ is a subspace of $(\scrF_0)_x\subset \scrF_0$ and we have the orthogonal splitting 
  \begin{equation}
    (\scrF_0)_x=\scrF_x\oplus \bbR \varphi_x
  \end{equation}
  where $\varphi_x\in \scrF_0$ is the function given for all $y\in M$ by 
  \begin{equation}
    \varphi_x(y):=f_0(x)\, \cov(x,y)-\cov(x) f_0(y).
  \end{equation}
  Moreover, we have $\|\varphi_x\|=\sqrt{\cov(x)\cov_0(x)}.$
\end{proposition}

\begin{proof}
  The fact that $\scrF_x\subset (\scrF_0)_x$ follows from the fact that they are both defined by the equation $f(x)=0$. First we note that, as a function $\cov(x,\cdot)\in \scrF$ and thus $f_0(x)\, \cov(x,\cdot)-\cov(x) f_0\in \scrF_0$. Moreover, since $\cov(x,\cdot)$ and $f_0$ are linearly independent in $\scrF_0$ and since $\cov(x)>0$, we have that $\varphi_x\neq 0$ (as an element of $\scrF_0$). 
  
  By substituting $y$ by $x$ in the definition of $\varphi_x$, we get $\varphi_x(x)=0$, i.e. $\varphi_x\in (\scrF_0)_x$. We now show that it is orthogonal to $\scrF_x$. Indeed, for every $f\in \scrF_x$, we have
  \begin{equation}
    \langle \varphi_x,f\rangle=f_0(x)\langle \cov(x,\cdot),f\rangle -\cov(x) \langle f_0,f\rangle.
  \end{equation}
  By definition of the scalar product on $\scrF_0$ and since $f\in \scrF_x\subset\scrF$, we have $\langle f_0,f\rangle=0$. Moreover, as previously observed in \eqref{eq:ev=K}, identifying $\scrF\cong\scrF^*$, we have $\cov(x,\cdot)=ev_x$. Thus we have $\langle \cov(x,\cdot),f\rangle=f(x)=0$ since $f\in \scrF_x$. We obtain $\langle \varphi_x,f\rangle=0$ and this proves the orthogonal splitting.

  It remains to compute the norm of $\varphi_x$. By definition of the norm on $\scrF_0$, we get
  \begin{equation}
    \|\varphi_x\|^2=f_0(x)^2\|\cov(x,\cdot)\|^2+\cov(x)^2=f_0(x)^2\cov(x)+\cov(x)^2
  \end{equation}
  where, in the second equality, we used the fact that that $\|\cov(x,\cdot)\|^2=\|ev_x\|^2=\cov(x)$. The result follows by factoring $\cov(x)$.
\end{proof}

\begin{remark}
  Note that $D_x(\varphi_x/\|\varphi_x\|)=(\cov(x)/\sqrt{\cov_0(x)})\tau(x)$. In fact, as claimed above, one can use \cref{prop:split Fox} to obtain another proof of \cref{lem:g0x} by mapping all the terms through $D_x$ and renormalizing carefully. This may be particularly useful if one wants to work towards a generalization when $\scrF$ is infinite dimensional.
 \end{remark}

Note that on the points $x\in M$ such that $f_0(x)=0$, the comparison of $\GRF$ and $\GRF_0$ becomes trivial. Thus we can assume $f_0(x)\neq 0$ without loss of generality. We prove that the form $\tau$ is \emph{integrable}.

\begin{proposition}\label{prop:alpha with Phio}
Suppose $f_0(x)\neq 0$ for all $x\in M$. Then we have for all $x\in M$:
\begin{equation}
  \tau(x)=\frac{1}{\sqrt{1+e^{-2\Phi_0(x)}}}\, D_x \Phi_0.
\end{equation} 
where 
\begin{equation}\label{eq:defPhio}
    \Phi_0(x):=\frac{1}{2}\log\left(\frac{\cov(x)}{f_0(x)^2}\right)=\Phi(x)-\log{|f_0(x)|}
\end{equation}
\end{proposition}

\begin{proof}
  This is a straightforward computation, indeed
  \begin{equation}
    D_x \Phi_0=D_x(\Phi-\log|f_0|)=\mu(x)- \frac{D_xf_0}{f_0(x)}=\frac{\sqrt{\cov_0(x)}}{f_0(x)}\tau(x).
  \end{equation}
  To conclude, it remains only to see that $\sqrt{\cov_0}/f_0=\sqrt{1+\frac{K}{f_0^2}}=\sqrt{1+e^{-2\Phi_0(x)}}$.
\end{proof}

\begin{remark}
  In the spirit of \cref{rk:gox bigger and smaller}, this proposition helps us to visualize what is happening. Indeed, in terms of ellipsoid, $\zeta_{\GRF_0}$ is \emph{elongated} with respect to $\zeta_\GRF$ in the direction $\tau$ and then shrunk by a coefficient $\sqrt{\cov/\cov_0}$. The previous proposition then tells us that the direction $\tau$ is \emph{orthogonal} to the level sets of the function $\Phi_0$, see \cref{fig:ellip} created using \verb|GeoGebra|. In particular, when projecting onto the tangent to these level sets, the projection of $\zeta_{\GRF_0}$ is \emph{contained} in the projection of $\zeta_\GRF$.
\end{remark}

\begin{figure}
  \includegraphics[scale=0.8]{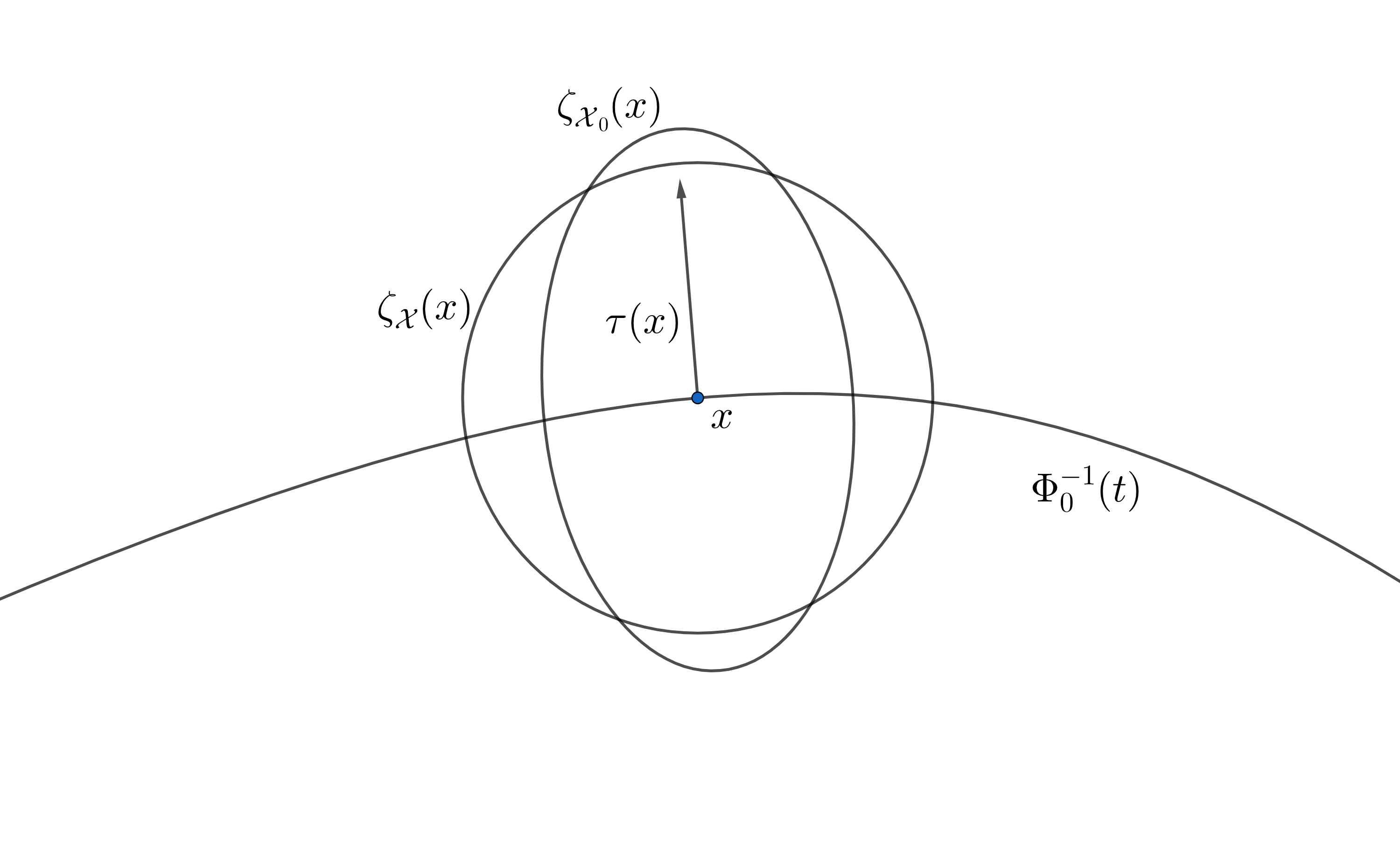}
      \caption{The ellipsoids of $\GRF$ and $\GRF_0$}
      \label{fig:ellip}
\end{figure}

We conclude this section by interpreting the last remark in terms of decreasing expected solution on the random sets restricted to a certain hypersurface. Recall that, for a smooth function $F:M\to N$, between smooth manifolds, $q\in N$ is a \emph{regular value} of $F$ if for all $p\in M$ such that $F(p)=q$, the differential $D_pF:T_pM\to T_qN$ is surjective. In particular, this guarantees that $F^{-1}(q)$ is a smooth submanifold of $M$ of codimension the dimension of $N$.
\begin{proposition}\label{prop:decreasHypGRF}
  Let $m>1$ and assume that $f_0(x)\neq 0$ for all $x\in M$ and let $t\in \bbR$ be a regular value of $\Phi_0$. Consider $S_t:=\Phi_0^{-1}(t)$ smooth hypersurface in $M$. For any admissible random fields $W\randin C^\infty(S_t,\bbR^{m-2})$ on $S_t$ and all open set $U\subset S_t$, we have:
  \begin{equation}
    \bbE\#\sol{\GRF_0|_{S_t},W; U}\leq \bbE\#\sol{\GRF|_{S_t},W; U}.
  \end{equation} 
   In particular
  \begin{equation}
    \esol{\GRF_0|_{S_t}; U}\leq \esol{\GRF|_{S_t}; U}.
  \end{equation} 
  In both cases, the inequality is an equality if and only if the right hand side is zero in which case every term is equal to zero.
\end{proposition}

\begin{proof}
  We write $S=S_t$ and fix $x\in S$. We will prove that $\zeta_{\GRF_0|_S}(x)\subset\zeta_{\GRF|_S}(x)$. By \cref{prop:pull back}, we have 
  \begin{equation}
    \zeta_{\GRF_0|_S}(x)=\pi_x(\zeta_{\GRF_0}(x)) \mathand \zeta_{\GRF_0|_S}(x)=\pi_x(\zeta_{\GRF_0}(x)) 
  \end{equation}
  where $\pi_x:T^*_xM\to T^*_xS$ is the projection dual to the inclusion $\iota: T_xS\to T_xM$ induced by the embedding. Now we have for all $v\in T_x S$
  \begin{equation}
    \langle\pi_x(\tau(x)),v\rangle=\langle \tau(x),v\rangle=e^{-\Phi_0(x)}\langle D_x\Phi_0,v\rangle=0
  \end{equation}
  where the second equality follows from the expression of $\tau$ in \cref{prop:alpha with Phio} and the third from the fact that $T_xS=D_x\Phi_0^\perp$. In other words, $\tau\in \ker(\pi_x)$. We write $B_0\subset T_xM$ the unit ball for the AT metric for $\GRF_0$ in such a way that $\zeta_{\GRF_0}(x)=B_0/(2\pi)$. Moreover, recalling \cref{eq:Bo as g}, we have 
  \begin{equation}
    \pi_x(\zeta_{\GRF_0}(x))=\frac{1}{2\pi}\sqrt{\frac{\cov(x)}{\cov_0(x)}}\pi_x(B_{g_x+\tau^2}^\circ)=\frac{1}{2\pi}\sqrt{\frac{\cov(x)}{\cov_0(x)}}\pi_x(B_{g_x}^\circ)
  \end{equation}
  where in the last equality we used \cref{prop: proj ker}. Recognizing $\zeta_\GRF(x)=\frac{1}{2\pi}B_{g_x}^\circ$, we then obtain:
  \begin{equation}
    \pi_x(\zeta_{\GRF_0}(x))=\sqrt{\frac{\cov(x)}{\cov_0(x)}}\pi_x(\zeta_{\GRF}(x)).
  \end{equation}
  Since $\cov/\cov_0<1$, we have that $\zeta_{\GRF_0|_S}(x)\subset\zeta_{\GRF|_S}(x)$ as claimed.

  Now we consider $\zeta_W(x)\subset\Lambda^{m-2}T^*_xS$ the zonoid associated to the admissible random field $W$. We have from the main Kac-Rice formula \cref{prop:mainZKROK}, for all $U\subset S$ open:
  \begin{equation}
    \bbE\#\sol{\GRF_0|_S,W;U}=\int_U\ell(\pi_x(\zeta_{\GRF_0}(x))\wedge \zeta_W(x))\mrd x=\int_U\sqrt{\frac{\cov(x)}{\cov_0(x)}}\ell(\pi_x(\zeta_{\GRF}(x))\wedge \zeta_W(x))\mrd x
  \end{equation}
  where we used the linearity of the first intrinsic volume $\ell$. The inequality follows again from the fact that $\cov/\cov_0<1$. Since the inequality is strict, the equality case can only occur when $\ell(\pi_x(\zeta_{\GRF}(x))\wedge \zeta_W(x))$ vanishes identically on $U$, which means that $\bbE\#\sol{\GRF|_S,W;U}=0$ and this concludes the proof.
\end{proof}

As a last remark, we note that whenever $\Phi_0$ is defined, the function $\Psi$ can be completely written in terms of $\Phi_0$. 
\begin{proposition}\label{prop:PsiPhio}
  Whenever $f_0(x)\neq 0$, we have
  \begin{equation}\label{eq:Psiinsecmon}
    \Psi(x)=\left(1- \frac{e^{-2\Phi_0(x)}}{1+e^{-2\Phi_0(x)}} \right)^{\frac{m}{2}}\sqrt{1+\frac{1}{1+e^{2\Phi_0(x)}} g^x\left(D_x \Phi_0\right)}.
  \end{equation}
\end{proposition}
\begin{proof}
  The second term in the product on the right hand side should be clear using \cref{prop:alpha with Phio}.
  To see the first term, first we recall that $e^{2\Phi_0}=\cov/f_0^2$ and then we write
  \begin{equation}
    \frac{\cov}{\cov_0}=\frac{\cov/f_0^2}{1+\cov/f_0^2}=\frac{e^{2\Phi_0}}{1+e^{2\Phi_0}}=1-\frac{1}{1+e^{2\Phi_0}}=1-\frac{e^{-2\Phi_0}}{1+e^{-2\Phi_0}}
  \end{equation}
  which is what we wanted.
\end{proof}

\section{Gaussian Exponential sums}

We now come to our main setting. We fix a finite subset $A\subset \bbR^m$ of size $d:=\#A$. Let $\alpha=(\alpha_a)_{a\in A}\in \bbR^A\cong\bbR^d$ be such that $\alpha_a>0$ for all $a\in A$. We build the following GRF on $M=\bbR^m$:
\begin{equation}\label{eq:def E(x)}
  \Exp(x):=\sum_{a\in A}\xi_a \alpha_a e^{\langle a, x\rangle}
\end{equation}
where $\xi_a\randin \bbR$ are iid standard Gaussian variables. We are now in the setting of the previous section with $\scrF$ being spanned by $f_a:=\alpha_a e^{\langle a,\cdot\rangle}$, $a\in A$. Note that $(f_a)_{a\in A}$ forms an orthonormal basis of $\scrF$ (for the scalar product induced by the covariance of $\Exp$).

Finally, we will denote
\begin{equation}
  P:=\conv(A)\subset \bbR^m
\end{equation}
and will call this the \emph{Newton polytope.} We will always assume that $\dim(P)=m$. This is no loss of generality as in the lower dimensional cases the corresponding system of equations would have no solutions generically, see \cref{rk:dim(P)=m} below.

\subsection{Random polynomials} \label{sec: poly}
Before starting our study, let us detail a little bit more how it relates to the study of random polynomials. Suppose $A\subset \bbN^n$ and write $w_i:=e^{x_i}$. Then, by considering the Gaussian exponential sum \cref{eq:def E(x)} as a function of $w=(w_1,\ldots,w_m)$, one gets the Gaussian polynomial:
\begin{equation}\label{eq: def P}
  \calP(w):=\sum_{a\in A}\xi_a \alpha_a w^a
\end{equation}
where $w^a=w_1^{a_1}\cdots w_m^{a_m}$. The image of the map $x\mapsto w$ is the positive orthant $\Rpos^m$. However, it is clear that $\calP$ is defined on the whole $\bbR^m$, or on $(\bbR\setminus\{0\})^m$ if we take $A\subset \bbZ^m$ instead.

Moreover, if we substitute $w_i$ by $-w_i$ for some $i\in \{1,\ldots,m\}$ in \cref{eq: def P}, it changes $\xi_a$ to $(-1)^{a_i} \xi_a$ which is still distributed as $\xi_a$. This means that the random polynomial $\calP$ is identically distributed in each orthant defined by some choice of sign of the coordinates. We can thus relate the expected number of zeros of random exponential sums and random polynomials. More precisely, if we write $\bbR^*:=\bbR\setminus\{0\}$, then we have:
\begin{equation}\label{eq: solE solP}
  \esol{\calP;\left(\bbR^*\right)^m}=2^m \esol{\Exp;\bbR^m}.
\end{equation}

\begin{remark}\label{rk:Gl inv}
  If $L\in Gl_m(\bbR)$ and if we change $A$ by $L(A)$, then the Gaussian exponential sum \cref{eq:def E(x)}, changes to $\Exp\circ L^T$. Since $L$ is invertible, this doesn't change the total number of zeros: $\esol{\Exp\circ L^T;\bbR^m}=\esol{\Exp;\bbR^m}$. Similarly for the corresponding system of random polynomials. 
\end{remark}

\subsection{Moment map and convex potential}
We start by computing the covariance function and potential from \cref{sec:GRFs and AT}. We have, for all $x\in \bbR^m:$
\begin{equation}
  \cov(x)=\sum_{a\in A} \alpha_a^2 e^{2\langle a, x\rangle}\quad \text{and} \quad \mu(x)=\sum_{a\in A}\lambda_a \, a
\end{equation}
where 
\begin{equation}
  \lambda_a:=\frac{\alpha_a^2 e^{2\langle a, x\rangle}}{\cov(x)}.
\end{equation}
Moreover, we will write
\begin{equation}
  \overline{\cov}(x):=e^{-2h_P(x)}\cov(x)
\end{equation}
where recall from \cref{sec:Convex} the support function $h_P$ of the Newton polytope. We recall also the notation $A^x$ and $P^x$ for the (exposed) face from \cref{sec:Convex}. Note that we have 
\begin{equation}\label{eq:Kbar}
  \sum_{a\in A}\lambda_a=1 \mathand \overline{\cov}(x)= \|\alpha^x\|^2 +\sum_{a\in A\setminus A^x} \alpha_a^2e^{-2(h_P(x)-\langle a, x\rangle)}
\end{equation}
where  $\|\alpha^x\|^2 :=\sum_{a\in A^x} \alpha_a^2.$

We can also write down explicitly the AT metric $g_x$ for $\Exp$ using \cref{eq:g_x in basis}. We obtain for all $x\in \bbR^m$:
\begin{equation}\label{eq:gx and alpha}
  g_x=\sum_{a\in A} \lambda_a (a-\mu(x))^2=\sum_{a\in A} \lambda_a a^2-\mu(x)^2
\end{equation} 

\begin{remark}\label{rk:dim(P)=m}
  Since $\mu$ is a convex combinations of points in $A$, it belongs to the Newton polytope $P$. It follows that the quadratic form $g_x$ given by \cref{eq:gx and alpha} is non degenerate if and only if $\dim(P)=m$. By \cref{prop:non deg equiv}, if this is not the case, i.e., if $\dim(P)<m$ the GRF $\Exp$ is degenerate and $\vol_m(\zeta_\GRF(x))=0$ which, in turn, implies that $\esol{\GRF;\bbR^m}=0$. Hence we always assume that $\dim(P)=m$. 
\end{remark}

Asymptotically, we have the following.
\begin{proposition}\label{prop:asymptotics}
  We have for all $x\in \bbR^m$:
  \begin{equation}
    \lim_{t\to +\infty}\overline{\cov}(tx)= \|\alpha^x\|^2  \mathand \lim_{t\to +\infty}\frac{1}{t}\Phi(tx)=h_P(x).
  \end{equation}
  Moreover,
  \begin{equation}\label{eq:muinfty}
    \lim_{t\to +\infty}\mu(tx)=\sum_{a\in A^x}\frac{\alpha_a^2}{\|\alpha^x\|^2 }a =: \mu_\infty^x.
  \end{equation}
\end{proposition}

\begin{proof}
  Note that, by definition of the support function,   $h_P(x)-\langle a, x\rangle>0$ for all $a\in A\setminus A^x$ and $x\neq 0$. Thus the first limit follows simply from \cref{eq:Kbar}. Moreover, the second limit follows from the first one and the fact that $\Phi(tx)=\tfrac{1}{2}\log(\overline{\cov}(tx))+th_P(x).$

  Finally, to prove the third one, write 
  \begin{equation}
    \mu(tx) = \sum_{a\in A}\frac{\alpha_a^2 e^{-2 t (h_P(x)-\langle a, x\rangle)} }{\overline{\cov}(tx)} a 
    = \sum_{a\in A^x}\frac{\alpha_a^2 }{\overline{\cov}(tx)} a +\sum_{a\in A\setminus A^x}\frac{\alpha_a^2 e^{-2 t (h_P(x)-\langle a, x\rangle)} }{\overline{\cov}(tx)} a.
  \end{equation}
  Using the first limit, we get that $\sum_{a\in A^x}\frac{\alpha_a^2 }{\overline{\cov}(tx)} a \to \sum_{a\in A^x}\frac{\alpha_a^2}{\|\alpha^x\|^2 }a$. Moreover, using again the fact that for all $a\in A\setminus A^x$, $h_P(x)-\langle a, x\rangle>0$, the second term goes to zero and this concludes the proof.
\end{proof}

\begin{remark}
Note that for all $x\in \bbR^m$, we have $ \mu_\infty^x \in P^x.$
We will say more about what happens at infinity in \cref{sec:atinfty}.
\end{remark}

An interesting phenomenon in this setting is the fact that the AT metric can be derived directly from the potential $\Phi$.

\begin{lemma}\label{lem:Phiconv}
  The potential $\Phi$ is strictly convex and we have for all $x\in \bbR^m$:
  \begin{equation}
    g_x= \frac{1}{2} D^2_x\Phi
  \end{equation}
  where $D^2_x\Phi$ is the Hessian of the function $\Phi$ considered here as a quadratic form.
\end{lemma}
\begin{proof}
  First of all, we note that it is enough to prove the equality of quadratic forms, since $g_x$ is positive definite, it then follows that $\Phi$ is strictly convex. We have
  \begin{equation}
    D^2_x\Phi=\frac{1}{2}\frac{D^2_x\cov}{\cov(x)}-\frac{1}{2}\left(\frac{D_x\cov}{\cov(x)}\right)^2=\frac{1}{2}\frac{D^2_x\cov}{\cov(x)}-2\mu(x)^2.
  \end{equation}
  Furthermore we have $\frac{D^2_x\cov}{\cov(x)}= 4 \sum_{a\in A} \frac{\alpha_a^2e^{2\langle a, x\rangle}}{\cov(x)} a^2=4 \sum_{a\in A} \lambda_aa^2.$
  Reintroducing in the second expression for $g_x$ in \cref{eq:gx and alpha}, it yields $g_x=\frac{1}{4}\frac{D^2_x\cov}{\cov(x)}-\mu(x)^2=\frac{1}{2} D^2_x\Phi$ which is what we wanted.
\end{proof}

The other remarkable property is the following, that is also one of the reasons for the name \emph{moment map}. We denote by $\interior(P)$ the interior of the Newton Polytope $P$.

\begin{lemma}\label{lem:muiso}
  The moment map $\mu$ is a diffeomorphism $\mu:\bbR^m\xrightarrow[]{\sim}\interior(P)$ where $\interior(P)$.
\end{lemma}

\begin{proof}
  First of all, consider the expression $\mu(x)=\sum_{a\in A} \lambda_a(x) a$. For all $x\in \bbR^m$ and $a\in A$, $\lambda_a(x)>0$ and $\sum_{a\in A}\lambda_a(x)=1$. Therefore $\mu(x)\in \interior(P)$. Moreover, $D_x\mu=D^2_x\Phi$ which we proved is always positive definite in \cref{lem:Phiconv}. Thus we only need to prove surjectivity.
  
  Since $\Phi$ is convex, we have from \cref{prop:LT C1} that the image of $\mu$ is convex. Thus it is enough to prove that the closure of the image of $\mu$ contains every vertex of $P$. A vertex is a $0$ dimensional face of $P$, i.e., $v$ is a vertex of $P$ if and only if $v\in A$ and there is $x\in \bbR^m\setminus\{0\}$ such that $A^x=\{v\}$. For such $x$, it follows from \cref{eq:muinfty}, that $\lim_{t\to+\infty}\mu(tx)=v$ and thus $v$ belongs to the closure of the image of $\mu$ and this concludes the proof.

\end{proof}
We obtain the following which is one of the main results of this paper.
\begin{theorem}\label{thm:esol is volg P}
  For all open set $U\subset \bbR^m$, we have
  \begin{equation}
    \esol{\mathcal{E};U}=\frac{1}{2^{\tfrac{m-2}{2}}s_m} \int_{\mu(U)}\sqrt{det D^2_p \Phi^*} \, \mrd p
  \end{equation}
  where $\Phi^*:\interior(P)\to \bbR$ is the Legendre transform of $\Phi$ and where $s_m$ denotes the $m$-dimensional volume of the unit sphere $S^m\subset \bbR^{m+1}$. In particular, we have, for $U=\bbR^m$:
  \begin{equation}\label{eq: esol volg P}
    \esol{\mathcal{E};\bbR^m}=\frac{1}{2^{\tfrac{m-2}{2}}s_m} \vol_g(P)
  \end{equation}
  where $\vol_g(P)$ denotes the Riemannian volume of $P$ for the metric given for all $p\in P$ by $D^2_p\Phi^*$.
\end{theorem}

\begin{proof}
  From \cref{prop:esolisATvol}, we have
  \begin{align}
    \esol{\mathcal{E};U}&=\frac{2}{s_m}\int_U \sqrt{\det g_x}\, dx \\
    &=\frac{2}{2^{\frac{m}{2}}s_m}\int_U \sqrt{\det D^2_x\Phi}\, dx \\
    &=\frac{1}{2^{\frac{m-2}{2}}s_m}\int_{\mu(U)} \frac{\sqrt{\det D^2_{\mu^{-1}(p)}\Phi}}{\det D^2_{\mu^{-1}(p)}\Phi}\, dp \\
    &=\frac{1}{2^{\frac{m-2}{2}}s_m}\int_{\mu(U)} \frac{1}{\sqrt{\det D^2_{\mu^{-1}(p)}\Phi}}\, dp 
  \end{align}
  where in the second equality we used \cref{lem:Phiconv} and for the third equality we applied the change of variable $p=\mu(x)$ remembering that $D_x\mu=D^2_x\Phi$. We conclude using the property of Legendre transform, see \cref{cor: LT Hessians}. 
\end{proof}

We deduce the following lower bound. We write $\diam(P)$ for the diameter of $P\subset \bbR^m$, that is 
\begin{equation}
  \diam(P):=\max\set{\|p_1-p_2\|}{p_1,p_2\in P}.
\end{equation}

\begin{corollary}\label{thm:lowerbound}
  We have 
  \begin{equation}\label{eq:lowerBound in diam}
    \esol{\Exp;\bbR^m}> \frac{1}{2^{m-1}s_m} \,\frac{\vol_m(P)}{\diam(P)^m}
  \end{equation}
  where $\vol_m(P)$ denotes the classical Euclidean volume of $P$.
\end{corollary}

\begin{proof}
  We use the explicit expression of the AT metric \cref{eq:gx and alpha} and Cauchy-Schwartz (more precisely \cref{prop:QCS}) to obtain that for all $x\in \bbR^m$ we have 
  \begin{equation}
    g_x\leq \sum_{a\in A}\lambda_a(x)\, \|a-\mu(x)\|^2< \diam(P)^2.
  \end{equation}
  where recall the meaning of the bound in \cref{def:Qleqr} and where, for the second inequality we used the fact that $a\in P$ and $\mu(x)\in\interior(P)$ for all $x\in\bbR^m$ and $a\in A$ and that $\sum_{a\in A}\lambda_a(x)=1$.

  Equivalently, using \cref{lem:Phiconv}, we have $D^2_x\Phi< 2 \diam(P)^2$. Using the property of Legendre transform \cref{cor: LT Hessians}, we have $D^2_{\mu(x)}\Phi^*=(D^2_x\Phi)^\circ$. By \cref{prop:Duality lambda 1/lambda}, we obtain for all $p\in P$, $D^2_p\Phi^*> 1/(2 \diam(P)^2)$ and thus $\sqrt{\det D^2_p\Phi^*}> (2 \diam(P))^{-\frac{m}{2}}.$
  Using \cref{eq: esol volg P}, we obtain
  \begin{equation}
    \esol{\mathcal{E};\bbR^m}=\frac{1}{2^{\tfrac{m-2}{2}}s_m} \int_P \sqrt{\det D^2_p\Phi^*}\, \mrd p >\frac{1}{2^{m-1}s_m \diam(P)^m} \int_P 1\, \mrd p
  \end{equation}
  which is what we wanted
\end{proof}

\begin{remark}
  Note that the left hand side in \cref{eq:lowerBound in diam} is affine invariant, meaning that, changing $A$ by an affine transformation does not change the expected total number of solutions (see \cref{rk:Gl inv}), but the right hand side is not. It would be interesting to understand the maximizers in an affine class. Maximizing the volume with a fixed diameter is known as the \emph{isodiametric problem}. It is a surprising fact that, in $\bbR^2$, among the polygons with $k$ vertices, the regular $k$-gon is \emph{not} optimal for $k>5$ and even, see \cite{Mossinghoff2006}. It is a natural and interesting question to ask if the exponential sums corresponding to these maximizers (see for example \cite[Figure~1]{Mossinghoff2006}) also maximize the expected total number of solutions among exponential sums with a support of size $k$. Although \cref{eq:lowerBound in diam} is not really a strong evidence in that direction since the use of Cauchy Schwartz in the proof is rather brutal and the inequality rather coarse.
\end{remark}

\subsection{Monotonicity}\label{sec:monot exp sums}
We now choose $a_0\in \bbR^m\setminus A$, and $\alpha_0=\alpha_{a_0}>0$ and let $f_0:=\alpha_0 e^{\langle a_0,\cdot\rangle}$. We define the exponential sum
\begin{equation}
  \Exp_0:=\Exp+\xi_0 f_0
\end{equation}
where $\xi_0\randin\bbR$ is a standard Gaussian variable independent of $\Exp$. The goal of this section is to use the general case of \cref{sec:locmonoGRF} to compare the local expected number of solutions $\esol{\Exp;U}$ and $\esol{\Exp_0;U}$ depending on the choice of $a_0$.

\begin{remark}\label{rk:ao=0}
  If we choose $b\in\bbR^m$ and shift $a\to a-b$ then we have $f_{a-b}=e^{-\langle b, \cdot \rangle}f_a$ for all $a\in A$ and similarly for $a=a_0$. Thus $\Exp$ changes to $e^{-\langle b, \cdot\rangle} \Exp$ and similarly for $\Exp_{0}$. Since $e^{-\langle b, \cdot\rangle}$ is a strictly positive function on $\bbR^m$, this does not affect the zero set and thus without loss of generality we can shift $A$ and $a_0$ by the same vector $b$ (in fact one can show that the zonoid section does not change when multiplying the random field by a positive function). Choosing $b=a_0$, this means that we can assume without loss of generality that $a_0=0$. We will sometimes do so, especially in proofs, as it makes computations easier. Nevertheless we will state our results for an arbitrary $a_0$. This has the advantage to be easier to interpret and easier to use in eventual explicit computations/applications.
\end{remark}

  Recall the definition of $\Phi_0$ in \cref{eq:defPhio}. Note that $\log f_0(x)=\log(\alpha_0)+\langle a_0, x\rangle$ is affine and thus $\Phi_0=\Phi-\log f_0$ is also convex and we have 
  \begin{equation}
    D_x\Phi_0=\mu(x)-a_0 \mathand D^2\Phi=D^2\Phi_0. 
  \end{equation}
We also recall the function $\Psi$ from \cref{sec:locmonoGRF} that determines the local monotonicity of the expected number of solutions and in particular the expression obtained in \cref{prop:PsiPhio}.

We define the following open sets:
\begin{equation}
  U_-:=\set{x\in\bbR^m}{\Psi(x)<1} \mathand U_+:=\set{x\in\bbR^m}{\Psi(x)>1}.
\end{equation}
It follows from \cref{thm:Psi}, that for all open set $U\subset U_-$, we have $\esol{\Exp_0;U}<\esol{\Exp;U}$ and for all $V\subset U_+$, $\esol{\Exp_0;V}>\esol{\Exp;V}$. The first result is the following.

\begin{theorem}\label{thm:a0 in int}
  If $a_0\in \interior(P)$ then $U_-$ is non empty. Therefore, there exists $U\subset \bbR^m$ such that
  \begin{equation}
    \esol{\Exp_0;U}<\esol{\Exp;U}.
  \end{equation}
\end{theorem}

\begin{proof}
  The proof is surprisingly simple at this point. From \cref{lem:muiso}, we know that $\mu$ is an isomorphism between $\bbR^m$ and $\interior(P)$. Thus there exists $x_0\in\bbR^m$ such that $\mu(x_0)=a_0$. We then have $g^{x_0}(\mu(x_0)-a_0)=0$. Reintroducing in \cref{eq:Psiinsecmon}, we get
  \begin{equation}
    \Psi(x_0)=\left(1- \frac{e^{-2\Phi_0(x_0)}}{1+e^{-2\Phi_0(x_0)}} \right)^{\frac{m}{2}}<1,
  \end{equation} 
  i.e., $x_0\in U_-$ and this concludes the proof.
\end{proof}
Next we prove a convenient characterization of $U_\pm$.
\begin{proposition}\label{prop:cherUmin}
  For all $x\in \bbR^m$, $x\in U_-$ if and only if 
  \begin{equation}\label{eq:cherUmin}
    g^x(\mu(x)-a_0)<m+\sum_{k=1}^{m-1} \binom{m+1}{k+1} e^{-2k\Phi_0(x)}+e^{-2m\Phi_0(x)}.
  \end{equation}
  Similarly $x\in U_+$ if and only if the reverse inequality holds.
\end{proposition}
\begin{proof}
  First of all we note that $\Psi(x)<1$ iff and only if $\Psi(x)^2<1$. From \cref{def:tau Psi}, we get:
  \begin{align}
    \Psi(x)^2<1 &\iff g^x(\mu(x)-a_0)<\left(\left(\frac{\cov_0(x)}{\cov(x)}\right)^m-1\right)\frac{\cov_0(x)}{f_0(x)^2} \\
    &\iff g^x(\mu(x)-a_0)<\left(\left(1+\frac{f_0(x)^2}{\cov(x)}\right)^m-1\right)\left(1+\frac{\cov(x)}{f_0(x)^2}\right) \\
    &\iff g^x(\mu(x)-a_0)<\sum_{k=1}^m \binom{m}{k} \left(\frac{f_0(x)^2}{\cov(x)}\right)^k+\sum_{k=1}^m \binom{m}{k} \left(\frac{f_0(x)^2}{\cov(x)}\right)^{k-1}.
  \end{align}
  To conclude, we substitute $ \frac{f_0(x)^2}{\cov(x)}=e^{-2\Phi_0(x)}$ and use the identity $\binom{m}{k}+\binom{m}{k+1}=\binom{m+1}{k+1}$. Since all the quantities are non negative, the same proof works replacing ``$<$'' by ``$>$''.
\end{proof}

We obtain the following.
\begin{theorem}\label{thm: a0 out}
  Let $a_0$ be such that $d(a_0,P)>\frac{\diam(P)}{m}$. Then $U_-$ is non empty and unbounded.
\end{theorem}

\begin{proof}
  We assume, without loss of generality, that $a_0=0$. Then, there is $\tilde{x}\in \bbR^m$ with $\|\tilde{x}\|=1$ such that $-h_P(\tilde{x})=d(0,P)$. By \cref{rk:NC vertex dense}, we can choose $x\in\bbR^m$ with $\|x\|=1$ such that $\dim(A^x)=0$ and $-h_P(x)>\diam(P)/m$. We fix such $x$, we write $\{a^x\}:=A^x$ and we will show that, for all $t>0$ big enough, we have $tx\in U_-$. 
  
  Since the support function $h_P$ is negative in $x$, and since $\Phi$ approximates the support function (see \cref{prop:asymptotics}), it follows that, for all $t>0$ big enough, $\Phi(tx)<0$. In that case, the right hand side in \cref{prop:cherUmin}, satisfies:
  \begin{equation}\label{eq:thiseq}
    m+\sum_{k=1}^{m-1} \binom{m+1}{k+1} e^{-2k\Phi_0(tx)}+e^{-2m\Phi_0(tx)} > c_1 e^{-2 m t h_P(x)}
  \end{equation}
  for some constant $c_1>0$.

  We want to show that, asymptotically, $g^{tx}(\mu(tx))$ is bounded from above by a term that grows slower than the right hand side in \cref{eq:thiseq}, then we could conclude using \cref{prop:cherUmin}. For this, we need a non trivial lower bound on the quadratic form $g_x$. Let us denote $\eta(tx):=a^x-\mu(tx)$. It follows from \cref{prop:asymptotics} that $\eta(tx)$ goes to zero. Moreover, the AT metric (see \cref{eq:gx and alpha}) satisfies: 
  \begin{equation}
    g_{tx}=\lambda_{a^x}(tx) \eta(tx)^2+\sum_{a\neq a^x}\lambda_a(tx)(a-a^x+\eta(tx))^2
    \geq \sum_{a\neq a^x}\lambda_a(tx)(a-a^x+\eta(tx))^2.
  \end{equation}

  Now let $\bar{a}\in A\setminus a^x$ be such that $\lambda_{\bar{a}}(tx)\leq \lambda_{a}(tx)$ for all $a\in A\setminus a^x$. Since $\lambda_a(tx)=\alpha_a^2e^{t\langle a, x\rangle}/\cov(tx)$, for $t$ big enough, $\bar{a}$ is a point in $A$ that is the furthest in the direction $-x$, i.e. $\bar{a}$ is any point in $A^{-x}$ and in particular does not depend on $t$. We note that:
  \begin{equation}\label{eq:abarax}
    h_P(x)-\langle \bar{a},x\rangle= \langle a^x-\bar{a},x\rangle \leq d(a^x,\bar{a})\leq \diam(P).
  \end{equation}
  Moreover we get 
  \begin{equation}
    g_{tx}\geq \lambda_{\bar{a}}(tx)\sum_{a\neq a_x}(a-a^x+\eta(tx))^2.
  \end{equation}
  Since $P$ is full dimensional, $\set{a^x-a}{a\in A\setminus a^x}$ spans the whole space and thus the quadratic form $\sum_{a\neq a_x}(a-a^x)^2$ is positive definite. Since $\eta(tx)$ goes to zero, it follows that, there is $c_2>0$ such that for all $t>0$ big enough we have
  \begin{equation}
    g_{tx}\geq c_2 \lambda_{\bar{a}}(tx).
  \end{equation}
  It follows that, for $t>0$ big enough,
  \begin{equation}
    g^{tx}(\mu(tx))\leq \frac{1}{c_2}\lambda_{\bar{a}}(tx)^{-1} \|\mu(tx)\|^2=\frac{\overline{\cov}(tx)\|\mu(tx)\|^2}{c_2} e^{2t(h_P(x)-\langle \bar{a},x\rangle)}.
  \end{equation}
  The quantity $\frac{\overline{\cov}(tx)\|\mu(tx)\|^2}{c_2}$ goes to a positive constant (again by \cref{prop:asymptotics}). Thus, using \cref{eq:abarax}, we get that there is $c_3>0$ such that for all $t>0$ big enough:
  \begin{equation}
    g^{tx}(\mu(tx))\leq c_3 e^{2t\diam(P)}.
  \end{equation}
  Since we chose $x$ such that $-mh_P(x)>\diam(P)$, we get that, for $t>0$ big enough
  \begin{equation}
    c_3 e^{2t\diam(P)}\leq c_1 e^{-2 m t h_P(x)}
  \end{equation}
  and this concludes the proof.
\end{proof}

\begin{remark}
  From this proof we can deduce a little bit more information on $U_-$. Indeed, we can see that for all $x$ such that $\|x\|=1$ and $-(m-1)h_P(x)>\diam(P)$, the ray $\bbR_{\geq0} x$  intersects $U_-$ in an unbounded component. To know exactly the intersections, that is to determine the constants $c_i$ seems harder. In particular, it is not clear to the author how to express $c_2$ explicitly. 
\end{remark}

We conclude this section by applying \cref{prop:decreasHypGRF} in our context to find a convex hypersurface where the expected number of solution decreases. For all $t\in \bbR$, we define the sublevel sets
\begin{equation}
  \Lambda_t:=\set{x\in\bbR^m}{\Phi_0(x)\leq t}.
\end{equation}
By \cref{lem:Phiconv}, $\Lambda_t$ is convex. Moreover, from the limit \cref{prop:asymptotics}, when $a_0\in\interior(P)$, $\frac{1}{t}\Lambda_t$ approximates and is contained in $(P-a_0)^\circ$. Writing $\partial\Lambda_t$ for its boundary, which is a smooth convex hypersurface of $\bbR^m$ and write $t_{min}$ for the minimal value attained by $\Phi_0$. Then we note that, by strict convexity, any $t>t_{min}$ is a regular value of $\Phi_0$. We obtain the following, as a direct consequence of \cref{prop:decreasHypGRF}.

\begin{proposition}\label{prop:decreasHypExp}
  Let $m>1$ and let $t>t_{min}$ and let $S_t:=\partial\Lambda_t$. For any $a_0\in\bbR^m$ and for any admissible random fields $W\randin C^\infty(S_t,\bbR^{m-2})$ on $S_t$ and all open set $U\subset S_t$, we have:
  \begin{equation}
    \bbE\#\sol{\Exp_0|_{S_t},W; U}\leq \bbE\#\sol{\Exp|_{S_t},W; U}.
  \end{equation} 
   In particular
  \begin{equation}
    \esol{\Exp_0|_{S_t}; U}\leq \esol{\Exp|_{S_t}; U}.
  \end{equation} 
  In both cases, the inequality is an equality if and only if the right hand side is zero in which case every term is equal to zero.
\end{proposition}

\subsection{Tensors and Aronszajn multiplication}\label{sec:aron}
In this section, we want to keep track of the coefficients in the exponential sums. For every finite subset $A\subset \bbR^m$, we write $\Rpos^A:=\set{(\alpha_a)_{a\in A}}{\alpha_a>0}$. As before, for every $\alpha\in \Rpos^A$, we define the exponential sum 
\begin{equation}
  \Exp_\alpha(\cdot):=\sum_{a\in A} \alpha_a \xi_a e^{\langle a, \cdot \rangle}
\end{equation}
where $\xi_a\randin \bbR$ are iid standard Gaussian variables. Similarly, we will indicate all the associated functions with a subscript: $\cov_\alpha,\Phi_\alpha,\mu_\alpha,$ etc ... 

We want to define operations on the coefficients which will in turn give operations on the exponential sums. First let $A\subset \bbR^m$ and $B\subset \bbR^n$ be finite sets. Then $A\times B$ is a finite subset of $\bbR^{m\times n}$. Note also that $\bbR^{A\times B}\cong \bbR^A\otimes\bbR^B$ and that the tensor product preserves the positive cones. In other words, given $\alpha\in\Rpos^A$ and $\beta\in \Rpos^B$ we define $\alpha\otimes \beta\in \Rpos^{A\times B}$ given for all $(a,b)\in A \times B$ by 
\begin{equation}
  (\alpha\otimes \beta)_{(a,b)}:=\alpha_a \beta_b.
\end{equation}

Then we define the Gaussian exponential sum $\Exp_\alpha\otimes\Exp_{\beta}$ over $\bbR^m\times\bbR^n$ given for all $(x,y)\in \bbR^m\times \bbR^n$ by:
\begin{equation}\label{eq:tensprodExp}
  \Exp_\alpha\otimes\Exp_{\beta} (x,y):=\Exp_{\alpha\otimes \beta} (x,y)=\sum_{a\in A,\, b\in B} \xi_{ab} \alpha_a \beta_b e^{\langle b, x \rangle} e^{\langle a, y \rangle}
\end{equation}
where $\xi_{ab} \randin \bbR$ are iid standard Gaussian variables.
\begin{remark}
  Note that technically we are not defining an operation on the random fields $\Exp_\alpha$ and $\Exp_{\beta}$ but rather on their law. That is, the random field $\Exp_\alpha\otimes\Exp_{\beta}$ is not a function of the random fields $\Exp_\alpha$ and $\Exp_{\beta}$ but an entirely new field whose \emph{law} depends on the law of $\Exp_\alpha$ and $\Exp_{\beta}$. In this paper we are only interested in the laws since we are considering only expectations. We shall therefore continue to abuse notation this way and denote the operations we define as operations on random fields although they are operations on the probability measures.
\end{remark}

Next we define another operation which originally appears in \cite{MALAJO}.

\begin{definition}[Aronszajn's multiplication] Let $A,B\subset \bbR^m$ be finite subsets and consider their Minkowski sum $A+B\subset \bbR^m$. For all $\alpha\in \Rpos^A$, $\beta\in \Rpos^B$, we define their \emph{Aronszajn multiplication} $\alpha\odot \beta\in \Rpos^{A+B}$ to be given for all $c\in A+B$ by 
  \begin{equation}
    (\alpha\odot\beta)_c:=\sqrt{\sum_{a+b=c}\alpha_a^2\beta_b^2}
  \end{equation}
  where the sum runs over all the $a\in A$ and $b\in B$ such that $a+b=c$.
  We also consider the corresponding operation on exponential sums $\Exp_\alpha\odot \Exp_{\beta}:=\Exp_{\alpha\odot\beta}$ that we also call Aronszajn's multiplication.
\end{definition}

The tensor product and Aronszajn's multiplication are commutative operations. We show next that they also commute with each other.

\begin{proposition}\label{prop:commut otimes odot}
  Let $A,A'\subset \bbR^m$ and $B,B'\subset \bbR^n$ be finite sets. For every $\alpha\in \Rpos^A$, $\alpha'\in \Rpos^{A'}$, $\beta\in \Rpos^{B}$ and $beta'\in \Rpos^{B'}$, we have
  \begin{equation}
    (\alpha\otimes \beta)\odot (\alpha'\otimes \beta')=(\alpha\odot \alpha')\otimes(\beta\odot\beta').
  \end{equation}
\end{proposition}
\begin{proof}
  This is just a straightforward computation. Indeed, let $(c_1,c_2)\in (A\times B)+(A'\times B')=(A+A')\times (B+B')$ and let $\gamma:=(\alpha\otimes \beta)\odot (\alpha'\otimes \beta')$. Then we have:
  \begin{align}
    \gamma_{(c_1,c_2)}&=\sqrt{\sum_{(a,b)+(a',b')=(c_1,c_2)}(\alpha\otimes\beta)_{(a,b)}^2(\alpha'\otimes\beta')_{(a',b')}^2} \\
    &=\sqrt{\sum_{a+a'=c_1} \sum_{b+b'=c_2} \alpha_{a}^2\, \beta_{b}^2\, (\alpha')_{a'}^2\, (\beta')_{b'}^2} \\
    &=(\alpha\odot\alpha')_{c_1}\, (\beta\odot\beta')_{c_2}
  \end{align}
  which is what we wanted.
\end{proof}

We can give an interpretation of Aronszajn's multiplication of exponential sums in term of tensor.

\begin{proposition}
  Let $A,B\subset\bbR^m$ be finite sets and let $\alpha\in \Rpos^A$ and $\beta\in \Rpos^B$. Consider the diagonal embedding $\Delta:\bbR^m\to \bbR^m\times\bbR^m$, that is $\Delta(x):=(x,x)$. Then, in law, we have the equality
  \begin{equation}
    \Exp_\alpha\odot\Exp_\beta=\left(\Exp_\alpha\otimes\Exp_\beta\right)\circ\Delta.
  \end{equation}
\end{proposition}
\begin{proof}
  For all $x\in \bbR^m$, we have 
  \begin{equation}
    \left(\Exp_\alpha\otimes\Exp_\beta\right)(x,x)=\sum_{a\in A,\,b\in B}\alpha_a\beta_b\xi_{ab}e^{\langle a+b,x\rangle} =\sum_{c\in A+B} \left(\sum_{a+b=c}\alpha_a\beta_b \xi_{ab}\right) e^{\langle c,x\rangle}.
  \end{equation}
  Then it remains only to see that $\sum_{a+b=c}\alpha_a\beta_b \xi_{ab}$ is distributed as $(\alpha\odot \beta)_c \xi_c$ where $\xi_c\randin\bbR$ is a standard Gaussian variable.
\end{proof}

According to \cref{thm:esol is volg P}, the expected number of solutions of Gaussian exponential sums $\esol{\Exp_\alpha;\cdot}$ can be understood from the convex potential $\Phi_\alpha$. We thus describe how this potential transforms under the operations we just defined.

\begin{proposition}\label{prop:potential under tens and Ar}
  Let $A\subset \bbR^m$ and $B\subset \bbR^n$ be finite subsets and let $\alpha\in\Rpos^A$ and $\beta\in \Rpos^B$. For all $(x,y)\in \bbR^m\oplus\bbR^n$ we have 
  \begin{equation}
    \Phi_{\alpha\otimes \beta}(x,y)=\Phi_\alpha(x)+\Phi_\beta(y).
  \end{equation}
  Moreover, if $n=m$, then for all $x\in\bbR^m$ we have
  \begin{equation}
    \Phi_{\alpha\odot \beta}(x)=\Phi_\alpha(x)+\Phi_\beta(x).
  \end{equation}
\end{proposition}

\begin{proof}
  We consider
  \begin{equation}
    \cov_{\alpha\otimes\beta}(x,y)=\sum_{(a,b)\in A\times B}\alpha_a^2\beta_b^2 e^{2\langle (a,b),(x,y)\rangle} =\sum_{a\in A}\sum_{b\in B}\alpha_a^2\beta_b^2 e^{2\langle a,x\rangle}e^{2\langle b,y\rangle} =\cov_{\alpha}(x) \cov_\beta(y).
  \end{equation}
  Taking the logarithm and dividing by $2$ gives the first equality.
  The second equality is done similarly. 
\end{proof}

\begin{remark}
  In Aronszajn's theory of RKHS, this operation corresponds to a product of reproducing kernels, see \cite[Section I-8]{AronszajnTRK}.
\end{remark}

We can then deduce properties of the expected number of solutions.

\begin{proposition}\label{prop:esol of tens}
  Let $A\subset \bbR^m$ and $B\subset \bbR^n$ be finite subsets and let $\alpha\in\Rpos^A$ and $\beta\in \Rpos^B$. For all $U\subset \bbR^m$, $V\subset\bbR^n$, we have 
  \begin{equation}
    \esol{\Exp_{\alpha}\otimes\Exp_{\beta};U\times V}=2\,\frac{s_ms_n}{s_{m+n}}\,\esol{\Exp_{\alpha};U}\esol{\Exp_{\beta};V}.
  \end{equation}
\end{proposition}

\begin{proof}
  The proof is a direct consequence of the previous proposition and the fact that the AT metric is the Hessian of the convex potential (\cref{lem:Phiconv}). Indeed, let us denote by $g$, respectively $g_{\alpha}$ and $g_\beta$  the AT metric on $\bbR^m\times \bbR^n$, respectively $\bbR^m$ and $\bbR^n$ of the GRF $\Exp_{\alpha}\otimes\Exp_{\beta}$, respectively of $\Exp_{\alpha}$ and $\Exp_{\beta}$. Then for all $(x,y)\in \bbR^m\times\bbR^n$ it follows from \cref{lem:Phiconv} that
  \begin{equation}
    g_{(x,y)}=(g_\alpha)_x\oplus (g_\beta)_y.
  \end{equation}
  Thus we have $\det g_{(x,y)}=\det (g_\alpha)_x \det (g_\beta)_y$. We then apply \cref{prop:esolisATvol} to get 
  \begin{align}
    \esol{\Exp_{\alpha}\otimes\Exp_{\beta};U\times V}&=\frac{2}{s_{m+n}}\int_{U\times V}\sqrt{\det (g_\alpha)_x}\sqrt{\det (g_\beta)_y}\,\mrd x\, \mrd y \\
    &=\frac{2}{s_{m+n}}\left(\int_{U}\sqrt{\det (g_\alpha)_x}\,\mrd x\right) \left( \int_{ V}\sqrt{\det (g_\beta)_y}\,\mrd y\right).
  \end{align}
  To conclude, we apply again \cref{prop:esolisATvol} to recognize $\esol{\Exp_{\alpha};U}$ and $\esol{\Exp_{\beta};V}$.
\end{proof}

\begin{remark}
  If we consider $\esol{\Exp;\cdot}$ as a measure, this proposition says that $\esol{\Exp_{\alpha}\otimes\Exp_{\beta};\cdot}=c\, \esol{\Exp_{\alpha};\cdot}\otimes\esol{\Exp_{\beta};\cdot}$ where on the right hand side we have the tensor product of measures and $c=2\frac{s_ms_n}{s_{m+n}}$.

\end{remark}

\begin{remark}
  Although the proof is quite simple the result is not entirely trivial since the GRF $\Exp_\alpha\otimes\Exp_\beta(x,y)$ is \emph{not} distributed as $\Exp_\alpha(x)\Exp_\beta(y)$ when taking $\Exp_\alpha$ and $\Exp_\beta$ independent. Indeed the product of independent Gaussian $\xi_a\xi_b$ is not a Gaussian variable and therefore not distributed as $\xi_{(a,b)}$. However, \cref{prop:esol of tens} tells us that the density of expected number of zeros is the same (up to a constant). 
\end{remark}

A direct consequence is the following special case.

\begin{corollary}\label{cor: tensor power}
  Let $A_1,\ldots,A_m\subset \bbR$ be finite sets, let $\alpha^i\in \Rpos^{A_i}$, $i=1,\ldots,m$ and let $\alpha:=\alpha^1\otimes\cdots\otimes\alpha^m\in \Rpos^A$ where $A=A_1\times\cdots \times A_m\subset \bbR^m$. Then for every open sets $U_1,\ldots,U_m\subset\bbR$, writing $U:=U_1\times\cdots\times U_m$, we have:
  \begin{equation}
    \esol{\Exp_\alpha;U}=\frac{m!\ball_m}{ 2^{m}}\prod_{i=1}^m\esol{\Exp_{\alpha^i};U_i}.
  \end{equation}
\end{corollary}

With the same proof as in \cref{prop:esol of tens}, we can describe the transformation of the AT metric under the Aronszajn multiplication.

\begin{lemma}\label{lem:Aronsz and AT}
  Let $A,B\subset \bbR^m$ be finite subsets and let $\alpha\in\Rpos^A$ and $\beta\in \Rpos^B$. Denote by $g$, respectively $g_{\alpha}$ and $g_\beta$  the AT metric on $\bbR^m$ of the GRF $\Exp_{\alpha}\odot\Exp_{\beta}$, respectively of $\Exp_{\alpha}$ and $\Exp_{\beta}$. Then for all $x\in \bbR^m$ we have
  \begin{equation}
    g_{x}=(g_\alpha)_x+ (g_\beta)_x.
  \end{equation}
\end{lemma}

For every positive integer $d$, we denote by $\Exp^{\odot d}$ the $d$-th power of $\Exp$ for the Aronszajn multiplication. As a direct consequence of the previous proposition,
we obtain the following which is \cite[Theorem 1.3]{MALAJO} (extended to admissible fields).

\begin{proposition}\label{prop: dth Ar power}
  Let $A\subset \bbR^m$ be a finite set and let $\alpha\in\Rpos^A$. For every admissible random field $\GRF\randin C^\infty(\bbR^m,\bbR^{m-1})$ independent of $\Exp_\alpha$, every positive integer $d$ and every $U\subset \bbR^m$, we have 
  \begin{equation}
    \bbE\# \sol{ \left( \Exp_\alpha\right)^{\odot d},\GRF; U}=\sqrt{d}\,\bbE \# \sol {\Exp_\alpha,\GRF; U}
  \end{equation}
\end{proposition}

\begin{proof}
  By \cref{lem:Aronsz and AT}, the AT metric $g_x$ associated to $\Exp_\alpha^{\odot d}$ is 
  \begin{equation}
    g_x=d\cdot(g_\alpha)_x=(g_\alpha)_x\circ \sqrt{d}
  \end{equation}
  where on the last term we wrote $\sqrt{d}$ for the linear operator that just multiplies by $\sqrt{d}$ (recall that $(g_\alpha)_x$ is a quadratic form). Since the associated zonoid is (up to a multiplicative constant) the dual of the unit ball for this metric, by \cref{eq: BQL dual}, the zonoids transform as
  \begin{equation}\label{eq:thisthisthiseq}
    \zeta_{\Exp_\alpha^{\odot d}}(x)=\sqrt{d}\cdot\zeta_{\Exp_\alpha}(x).
  \end{equation}
  It remains to apply the Kac-Rice formula, that is \cref{prop:mainZKROK} to get 
  \begin{equation}
    \bbE\# \sol{ \left( \Exp_\alpha\right)^{\odot d},\GRF; U}=\int_U \ell(\zeta_{\Exp_\alpha^{\odot d}}(x)\wedge \zeta_\GRF(x))\mrd x.
  \end{equation}
  We conclude by \cref{eq:thisthisthiseq}, by linearity of $\ell$ and multilinearity of the wedge product. 
\end{proof}

\begin{example}[Kostlan polynomial]
  Let $A:=\{0,e_1,\ldots,e_m\}\subset\bbR^m$ and $\alpha\equiv 1$. Then the corresponding Gaussian Exponential sum is of the form
  \begin{equation}
    \Exp(x)=\xi_0+\sum_{i=1}^m\xi_i e^{x_i}
  \end{equation}
  where $\xi_0,\ldots,\xi_m\randin \bbR$ are iid standard Gaussian. If we write $w_i:=e^{x_i}$, we obtain the random polynomial $\calP(w):=\xi_0+\sum_{i=1}^m\xi_i w_i$, which in that case is just an affine function. It follows that 
  \begin{equation}
    \esol{\Exp;\bbR^m}=\frac{1}{2^m}\esol{\calP;\bbR^m}=\frac{1}{2^m}.
  \end{equation}
  Now we consider the $d$-th power of the Aronszajn multiplication. We have 
  \begin{equation}
  A^{+d}=\set{a\in \bbN^m}{|a|:=a_1+\cdots+a_m\leq d}
  \end{equation} 
  and $(\alpha^{\odot d})_a=\sqrt{\binom{d}{a}}$ where $\binom{d}{a}:=\frac{d}{a_1!\cdots a_m! (d-|a|)!}$. If we consider again the Gaussian exponential sum $\Exp^{\odot d}$ as a function of $w$, we obtain the \emph{Kostlan polynomial}:
  \begin{equation}\label{eq:Kostlanpol}
    \calP^{d}(w):=\sum_{|a|\leq d} \sqrt{\binom{d}{a}} \xi_a w^a
  \end{equation}
  where $\xi_a\randin\bbR$ are iid standard Gaussian variables and where $w^a:=w_1^{a_1}\cdots w_m^{a_m}$. We then obtain, from \cref{prop: dth Ar power}, Kostlan's celebrated result, see \cite{Kostlan1993}:
  \begin{equation}
    \esol{\calP^d;\bbR^m}={2^m}\esol{\Exp^{\odot d};\bbR^m}={2^m}\frac{(\sqrt{d})^m}{2^m}=d^{\frac{m}{2}}.
  \end{equation}
\end{example}

We conclude with the following bounds. The upper bound was already obtained by Malajovitch in \cite[Theorem~1.6]{MALAJO} but the lower bound is new.

\begin{theorem}\label{thm:uplowAron}
  Let $A_1,\ldots,A_n\subset \bbR^m$ be finite sets, let $\alpha_j\in\Rpos^{A_j}$ for $j=1,\ldots,n$, let $d_1,\ldots,d_n$ be positive integers and write $d:=\sum_{j=1}^n d_j$. Then for every admissible random field $\GRF\randin C^\infty(\bbR^m,\bbR^{m-1})$ independents of all the exponential sums, and every open set $U\subset \bbR^m$ we have 
  \begin{equation}
    \sum_{j=1}^n \frac{d_j}{\sqrt{d}}\, \bbE\#\sol{\Exp_{\alpha_j},\GRF;U} \leq \bbE\# \sol{\Exp_{\alpha_1}^{\odot d_1}\odot \cdots \odot \Exp_{\alpha_n}^{\odot d_n}, \GRF; U}\leq \sum_{j=1}^n \sqrt{d_j}\, \bbE\#\sol{\Exp_{\alpha_j},\GRF;U} 
  \end{equation}
\end{theorem}
\begin{proof}
  Let $g$ be the AT metric for the GRF $\Exp_{\alpha_1}^{\odot d_1}\odot \cdots \odot \Exp_{\alpha_n}^{\odot d_n}$, then by \cref{lem:Aronsz and AT}, we have, for all $x\in \bbR^m$, $g_x=\sum_{j=1}^n d_j (g_{\alpha_j})_x$. The result follows by applying \cref{prop:2-sum inclu}.
\end{proof}

\subsection{Behavior at infinity and Veronese map}\label{sec:atinfty}
In this section, we study the quantities that arise when studying the limit $t x\to \infty$ for a fixed $x\in \bbR^m$ and for $t$ going to $+\infty$. From the point of view of polynomials, this limit is the \emph{tropical limit}. Notice that we already have \cref{prop:asymptotics} in that direction. We also introduce, as in \cite{MALAJO}, a real Veronese map that gives an alternative interpretation of the AT metric as a pull back of the round metric on the sphere.

This section is more of an extended remark and do not contain very deep result. However we think this direction deserves attention and, in particular, it would be interesting to investigate deeper connections to tropical geometry.

We fix a finite subset $A\subset\bbR^m$ and a list of coefficients $\alpha\in \Rpos^A$. For all $x\in\bbR^m$, let us denote by $\zeta(x):=\zeta_{\Exp_\alpha}(x)$ the ellipsoid associated to the GRF $\Exp_\alpha$ in $x$. 
The first result is the following.

\begin{proposition}
  For all $x\in \bbR^m$, the limit $\zeta_\infty^x:=\lim_{t\to +\infty} \zeta(tx)$ exists and is contained in a vector space parallel to the affine span of $P^x$.
\end{proposition}

\begin{proof}
  The proof is a straightforward computation. Indeed, we have:
  \begin{equation}
    h_{\zeta(tx)}^2=\frac{1}{2\pi} g_x=\frac{1}{2\pi}\sum_{a\in A} \alpha_a^2 \frac{e^{-2t(h_P(x)-\langle a, x\rangle)}}{\overline{\cov}(tx)} (a-\mu(tx))^2.
  \end{equation}
  Now we apply \cref{prop:asymptotics} to obtain, for all $u\in \bbR^m$,
  \begin{equation}\label{eq:inf elli}
    \lim_{t\to\infty}h_{\zeta(tx)}(u)=\frac{1}{\sqrt{2\pi}}\sqrt{\sum_{a\in A} \frac{\alpha_a^2}{\|\alpha^x\|^2}\langle a-\mu_\infty^x, u\rangle^2}.
  \end{equation}
  This implies that the support function converges \emph{pointwise} to the support function of the ellipsoid defined by \cref{eq:inf elli} that we then denote $\zeta_\infty^x$. Pointwise convergence of support functions is actually enough to imply convergence in the Hausdorff distance, see \cite[ Theorem 1.8.15]{bible}. Now to see the last part of the statement, recall that $\mu_\infty^x\in P^x$ thus, if $a\in A^x$, for every $y$ orthogonal to $P^x$, we have $\langle a^x-\mu^x_\infty, y\rangle=0$. It implies that $\zeta_\infty^x$ is contained in a vector space parallel to $P^x$ which is what we wanted.
\end{proof}

We can actually see more precisely how the $AT$ metric degenerates at infinity.

\begin{proposition}
  Let $x\in \bbR^m$ and let $y\in x^\perp$. Then we have:
  \begin{equation}
    \lim_{t\to+\infty}(g_\alpha)_{y+tx}=(g_{\alpha^x})_y
  \end{equation}
  where recall that $\alpha^x$ is the restriction of $\alpha$ to $A^x$ and where the convergence of quadractic form is the uniform convergence as functions on the sphere.
\end{proposition}

\begin{proof}
  For every $a\in A$, we have 
  \begin{equation}
    \lambda_a(y+tx)=\frac{\alpha_a^2 e^{2\langle a, y\rangle}e^{-2t(h_P(x)-\langle a, x\rangle)}}{\sum_{a'\in A}\alpha_{a'}^2 e^{2\langle a', y\rangle}e^{-2t(h_P(x)-\langle a', x\rangle)}}.
  \end{equation}
  As $t\to \infty$ we get
  \begin{equation}
    \lambda_a(y+tx)\to\begin{cases}
      \lambda^x_a(y):=\frac{\alpha_a^2 e^{2\langle a, y\rangle}}{\sum_{a'\in A}\alpha_{a'}^2 e^{2\langle a', y\rangle}} & \text{if } a\in A^x\\
      0 & \text{if } a\in A\setminus A^x.
    \end{cases}
  \end{equation}
  It follows that $\mu_\alpha(y+tx)\to \mu_{\alpha^x}(y)$. We obtain pointwise convergence, that is, for every $v\in \bbR^m$, we have 
  \begin{equation}
    \lim_{t\to+\infty}(g_\alpha)_{y+tx}(v)=(g_{\alpha^x})_y(v).
  \end{equation}
  To deduce the uniform convergence (in $v$), we can consider these quadratic forms as (the square of) support functions of ellipsoid and use, as in the previous proof, the correspondence between pointwise and uniform convergence of support functions, that is \cite[ Theorem 1.8.15]{bible}. This concludes the proof.
\end{proof}

\begin{remark}
  When changing coordinates with the moment map, the AT metric is a Riemannian metric on the interior of the Newton polytope $P$. The previous proposition tells us that, when we approach the boundary, it degenerates to a metric on the interior of the face $P^x$ and that this metric is precisely the AT metric for the Gaussian exponential sum corresponding to the choice of coefficients $\alpha^x$.
\end{remark}

We conclude this section by introducing a map that strengthens the parallel with toric geometry. It was also considered by Malajovitch in \cite[Section~3]{MALAJO}. We write $S(\bbR^A):=\set{(v_a)_{a\in A}}{\sum_{a\in A}v_a^2=1}$ the unit sphere of $\bbR^A$ and $S(\Rpos^A):=S(\bbR^A)\cap \Rpos^A$. 

\begin{definition}
  The (real) \emph{Veronese map} is the map $\nu_\alpha:\bbR^m\to S(\Rpos^A)$ given for every $x\in \bbR^m$ by:
  \begin{equation}
    \nu_\alpha(x):=\left(\sqrt{\lambda_a(x)}\right)_{a\in A}.
  \end{equation}
\end{definition}

\begin{remark}
  Note that we could as well consider the image in the projective space $\bbP(\bbR^A)$ instead of the sphere. Since the hypersurface $S(\Rpos^A)$ is diffeomorphic (or isometric for the appropriate metric) to its image on the projective space, it makes no difference.
\end{remark}

We have the following
\begin{proposition}
  The Veronese map is an embedding and the AT metric on $\bbR^m$ is the pull back of the standard spherical metric on $S(\bbR^A)$.
\end{proposition}

\begin{proof}
  Consider the map $\pi:S(\Rpos^A)\to \interior(P)$ given by $\pi(v):=\sum_{a\in A} v_a^2 a$. Then the moment map is given by 
  \begin{equation}\label{eq:moment and veronese}
    \mu_\alpha=\pi\circ\nu_\alpha.
  \end{equation}
   Since the moment map is a diffeomorphism, it follows that $\nu_\alpha$ must be an embedding.

  We now compute its differential. A straightforward computation gives us
  \begin{equation}
    D_x (\sqrt{\lambda_a})=\sqrt{\lambda_a(x)}(a-\mu(x)).
  \end{equation}
  Now if we let $g_0$ be the round metric on $S(\bbR^A)$, and $e_a$, $a\in A$ the standard basis of $\bbR^A$, we have for every $u\in\bbR^m$:
  \begin{equation}
    (\nu_\alpha^*g_0)(u)=g_0(D_x\nu_{\alpha}u)=\left\|\sum_{a\in A} \langle D_x(\sqrt{\lambda_a}),u\rangle e_a\right\|^2=\sum_{a\in A}\lambda_a(x) \langle a-\mu(x),u\rangle^2=(g_\alpha)_x(u)
  \end{equation}
  and this concludes the proof.
\end{proof}

The image of the Veronese map was called by Malajovitch the \emph{real toric variety} \cite[Section~3]{MALAJO}. We see here that the metric properties of this submanifold are relevant to the study of Gaussian exponential sums. Moreover, from \cref{eq:moment and veronese}, we see that its closure can be identified with the Newton polytope $P$. This is, again, a striking similarity with the complex case and toric geometry. 
In fact this point of view when $m=1$ was adopted by Edelman and Kostlan in \cite{EdelmanKostlan} where they show that the expected number of zeroes of a random polynomial is the length of the (spherical) \emph{moment curve}, which corresponds exactly to the Veronese map just defined. This method was later applied by Jindal, Pandey, Shukla and Zisopoulos to study random univariate \emph{sparse} polynomials in \cite{EK-II}.

\subsection{An explicit example}\label{sec:example}

In this section we illustrate the various results obtained in the previous sections in examples where explicit computations are possible. 

Let us first consider the simplest setting with the set $A=\{0,1\}\subset \bbR$ and the coefficients $o_0:=1=:o_1$. The corresponding exponential sum
\begin{equation}
  \Exp_o(x)=\xi_0+\xi_1 e^x
\end{equation}
with $\xi_0,\xi_1\randin\bbR$ iid standard Gaussian variables, have a unique zero $x=-\log\left(\frac{\xi_1}{\xi_0}\right)$ if and only if $\xi_1/\xi_0>0$ which happens with probability $1/2$. Thus we have 
\begin{equation}\label{eq:sol Eo is 1/2}
  \esol{\Exp_o;\bbR}=\frac{1}{2}.
\end{equation}
Equivalentely, considering the polynomial viewpoint (see \cref{sec: poly}), we have $\calP_o(w):=\xi_0+\xi_1 w$ which is a random affine function and has almost surely one zero. This is coherent with \cref{eq: solE solP}.

We can compute the different quantities associated to this exponential sums, starting with the covariance function that is given for all $x\in \bbR$ by 
\begin{equation}
  \cov_o(x)=1+e^{2x}=2e^x  \cosh(x)
\end{equation}
where $\cosh(x):=\frac{e^x+e^{-x}}{2}$. It follows that the potential is given by 
\begin{equation}
  \Phi_o(x)=\frac{1}{2}\log\left(\cosh(x)\right)+\frac{x}{2}+\frac{1}{2}\log(2).
\end{equation}
It yields:
\begin{equation}
  \mu_o(x)=\Phi_o'(x)=\frac{1}{2}\tanh(x)+\frac{1}{2} \mathand g_x=\frac{1}{2}\Phi_o''(x)=\frac{1}{4}\frac{1}{\cosh(x)^2}
\end{equation}
where $\tanh:=\sinh/\cosh$ with $\sinh(x):=(e^x-e^{-x})/2$. 

We can use the operations of \cref{sec:aron} to have something more interesting. Consider the set: 
\begin{equation}
C^d_m:=\{0,\ldots,d\}^m  \subset \bbR^m.
\end{equation} 
Notice that we have $C_m^d:=(A\times\cdots\times A)^{+d}\subset \bbR^m$ where the product inside the bracket is repeated $m$ times. We consider the coefficients $\alpha^d_m:=(o^{\otimes m})^{\odot d}$. We can compute them explicitely and obtain for every $c\in C^d_m$: $(\alpha^d_m)_c=\binom{d}{c_1}^{\frac{1}{2}}\cdots\binom{d}{c_m}^{\frac{1}{2}}$. Indeed, we can use the commutativity between $\otimes$ and $\odot$ (\cref{prop:commut otimes odot}) and first compute for all $k\in \{0,\ldots,d\}$, $(o^{\odot d})_k=\sqrt{\binom{d}{k}}$ and then simply take the $m$-th tensor power.

We let, for all $x\in \bbR^m$: 
\begin{equation}
  \Exp^d_m(x):=(\Exp_o^{\otimes m})^{\odot d} (x)=\sum_{c\in C^d_m} \binom{d}{c_1}^{\frac{1}{2}}\cdots\binom{d}{c_m}^{\frac{1}{2}} \xi_c e^{\langle c, x\rangle}
\end{equation}
where, as usual, $\xi_c\randin\bbR$ are iid standard Gaussian variables. We also define the corresponding random polynomial, given for every $w\in\bbR^m$ by (compare with the Kostlan polynomial \cref{eq:Kostlanpol}):
\begin{equation}
  \calP^d_m(w):=\sum_{c\in C^d_m} \binom{d}{c_1}^{\frac{1}{2}}\cdots\binom{d}{c_m}^{\frac{1}{2}} \xi_c w^c.
\end{equation}
We can then, without doing any computation, obtain the expected number of solutions.

\begin{proposition}
  We have for all integers $m,d_1,\ldots,d_m>0$:
  \begin{equation}
    \bbE\#\sol{\calP^{d_1}_m,\ldots,\calP^{d_m}_m;\bbR^m}=\frac{m! b_m}{2^m} \, \sqrt{d_1\cdots d_m}.
  \end{equation}
\end{proposition}
\begin{proof}
  We switch to the point of view of exponential sums to get
  \begin{align}
    \bbE\#\sol{\calP^{d_1}_m,\ldots,\calP^{d_m}_m;\bbR^m} &=2^m\bbE\#\sol{\Exp^{d_1}_m,\ldots,\Exp^{d_m}_m;\bbR^m} \\
    &=2^m \sqrt{d_1\cdots d_m} \esol{\Exp^1_m;\bbR^m}\\
    &= m!b_m \sqrt{d_1\cdots d_m} \left(\esol{\Exp_o;\bbR}\right)^m \\
    &=\frac{m!b_m}{2^m}\sqrt{d_1\cdots d_m}
  \end{align}
  where in the second inequality we used \cref{prop: dth Ar power}, in the third \cref{cor: tensor power}, and in the last \cref{eq:sol Eo is 1/2} and this concludes the proof. 
\end{proof}

We can use the rules of transformation of the convex potential \cref{prop:potential under tens and Ar} to deduce the moment map $\mu_m^d$ and AT metric $(g_m^d)_x$ associated to the exponential sum $\Exp^d_m$:
\begin{equation}
  \mu_m^d(x)=\frac{d}{2}\sum_{i=1}^m \left( \tanh(x_i)+1\right)e_i \mathand (g_m^d)_x=\frac{d}{4}\sum_{i=1}^m\frac{1}{\cosh(x_i)^2}e_i^2.
\end{equation}
In this example, they are easily invertible.
Thus, choosing some $a_0$, we can express the function $\Psi$ in the $p$ coordinates explicitly. We collect different examples in \cref{fig:animals} where we identified the open set $U_-$. The figures where generated using \verb|SageMath 10.5|\cite{sagemath}. It is noticable that it is always connected which we expect to be true in general. Moreover, the case $a_0=(0.5,1.2)$ offers an example where $0<d(P,a_0)$ and $U_-$ is bounded.

\begin{figure}
  \centering
  \begin{subfigure}[b]{0.3\textwidth}
      \includegraphics[width=\textwidth]{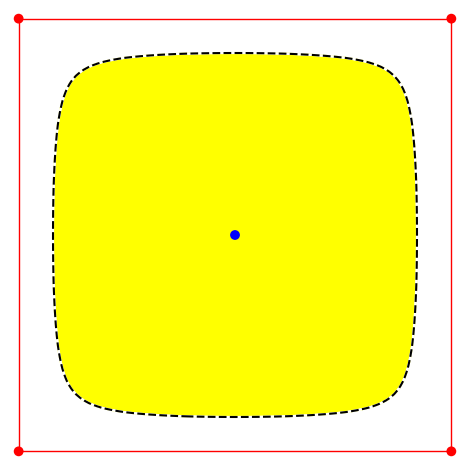}
      \caption{$a_0=(0.5,0.5)$}
      \label{fig:gull}
  \end{subfigure}
  ~ 
  \begin{subfigure}[b]{0.3\textwidth}
      \includegraphics[width=\textwidth]{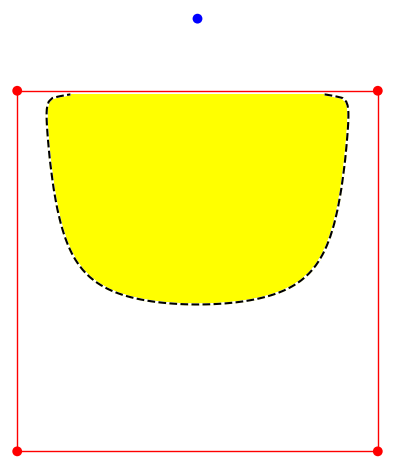}
      \caption{$a_0=(0.5,1.2)$}
      \label{fig:tiger}
  \end{subfigure}
  ~ 
  \begin{subfigure}[b]{0.3\textwidth}
      \includegraphics[width=\textwidth]{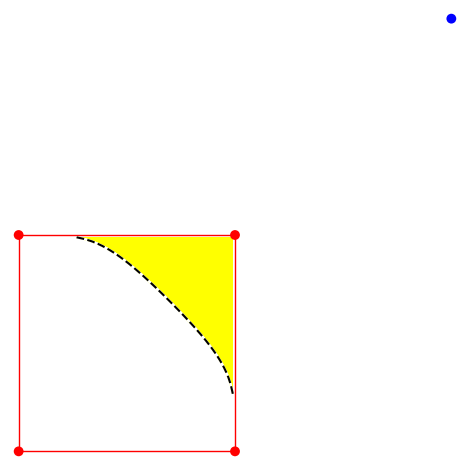}
      \caption{$a_0=(2,2)$}
      \label{fig:mouse}
  \end{subfigure}
  \caption{The set $U_-$ (yellow) in the Newton Polytope (red) for $d=1$ and different cases of $a_0$ (blue)}\label{fig:animals}
\end{figure}

\section{Complex Gaussian Exponential sums and the BKK Theorem}\label{sec:Cpx}
The aim of this section is to give a proof of the BKK Theorem and its generalization \cref{mainthm:Kaza},  using Gaussain Random Fields and our framework. Let us start with some definitions.

\begin{definition}
  A Gaussian vector $\xi\randin\bbC$ is called a \emph{standard complex Gaussian} if it is $U(1)$ invariant and satisfy $\bbE\xi=0$ and $\bbE[|\xi|^2]=1$. Equivalently, $\xi$ is distributed as $\tfrac{1}{\sqrt{2}}(\xi^R+i \xi^I)$ where $\xi^R,\xi^I\randin \bbR$ are standard (real) Gaussian variables. Moreover, we call $\xi\randin\bbC^n$ a \emph{standard complex Gaussian vector} if it is distributed as $(\xi_1,\ldots,\xi_n)$ where $\xi_1,\ldots,\xi_n\randin \bbC$ are iid standard complex Gaussians. Finally, a (centered) \emph{complex Gaussian vector} is the image of a standard complex Gaussian vector by a $\bbC$-linear map.
\end{definition}

\begin{remark}
  Note that, when identifying $\bbC\cong\bbR^2$ a standard complex Gaussian is \emph{not} a standard Gaussian vector of $\bbR^2$ because of the coefficient $1/\sqrt{2}$. More precisely, a standard complex Gaussian vector in $\bbC^n$ has density $z\mapsto \frac{1}{\pi^n} e^{|z|^2}.$
\end{remark}

We fix $A\subset \bbC^n$ finite and as before write $P:=\conv(A)$ for the Newton polytope. We now want to study the complex Gaussian exponential sum:
\begin{equation}
  \Exp^\bbC(z):=\sum_{a\in A}\xi_a f^\bbC_a(z)
\end{equation}
where $\xi_a\randin\bbC$ are independent standard complex Gaussians and $f^\bbC_a(z):=\alpha_a e^{(a,z)}$ with $(\cdot,\cdot):\bbC^n\times\bbC^n\to \bbC$ is the standard Hermitian form and $\alpha_a>0$. 

\begin{remark}
  Note that, if $\xi\randin \bbC$ is a standard Gaussian and $c\in \bbC$. Then $c\xi$ is distributed as $|c|\xi$ thus the fact that we take real positve coefficients in the exponential sum is not a restriction.
\end{remark}

As before, we define the following quantities for all $z\in \bbC^n$:
\begin{align}
\cov^\bbC(z)&:=\sum_{a\in A} |f^\bbC_a(z)|^2 & \lambda^\bbC_a(z)&:=\frac{|f^\bbC_a(z)|^2}{\cov^\bbC(z)} \\
\Phi^\bbC(z)&:=\log\left(\cov^\bbC(z)\right) & \mu^\bbC(z)&:=\sum_{a\in A} \lambda^\bbC_a(z)\, a.
\end{align}

We want to prove the following.

\begin{theorem}\label{thm:BKKprob}
  For all $U\subset \bbC^n$, we have 
  \begin{equation}\label{eq:deltaisdeldelbar}
    \esol{\Exp^\bbC;U}=\frac{n!}{\pi^n}\int_U \det\left(\del_z\delbar_z\Phi^\bbC\right)\, \mrd z.
  \end{equation}
  In particular, if $A\subset \bbZ^n$, then for all $V\subset\bbR^n$, we have 
  \begin{equation}
    \esol{\Exp^\bbC;V\times i (-\pi,\pi)}=n!\vol_n\left(\mu^\bbC(V)\right).
  \end{equation}
  Which gives in the case $V=\bbR^n$: 
  \begin{equation}
    \esol{\Exp^\bbC;\bbR^n\times i (-\pi,\pi)}=n!\vol_n(P).
  \end{equation}
\end{theorem}

We will of course apply the general theory introduced in \cref{sec:KR}. We let
\begin{equation}
  \scrF:=\Span_\bbC\set{f^\bbC_a}{a\in A}=\Span_\bbR\set{f^\bbC_a,i f^\bbC_a}{a\in A}\subset C^\infty(\bbC^n,\bbC).
\end{equation}
We endow this space with the Hermitian product such that $\{f^\bbC_a\}_{a\in A}$ is a unitary basis that we denote $(\cdot,\cdot)$. As before, we write $\scrF_z:=\set{f\in\scrF}{f(z)=0}$. Note that this is now a \emph{complex} hyperplane and thus of real codimension $2$. However the situation remains similar as in the real case, in pareticular, we have the following.

\begin{proposition}\label{prop:Gz}
  Let $z\in \bbC^n$ be fixed. For every $F:\scrF_z\to \bbR$ measurable, we have 
  \begin{equation}
    \int_{\scrF_z} F(f) \rho_{\Exp^\bbC}(f)\, \mrd f=\frac{1}{\pi}\bbE[F(\calG_z)] 
  \end{equation}
  where $\rho_{\Exp^\bbC}(f) :=(1/\pi^n)\exp(-|f|^2)$ is the density of $\Exp^\bbC\randin \scrF$ and where
  \begin{equation}\label{eq:defGz}
    \calG_z:=\Exp^\bbC-\frac{\Exp^\bbC(z)}{\cov^\bbC(z)}\cov^\bbC(z,\cdot)\randin \scrF_z
  \end{equation}
  with
  \begin{equation}
    \cov^\bbC(z,z'):=\sum_{a\in A}\overline{f^\bbC_a(z)}\, f^\bbC_a(z'), \quad \forall z'\in\bbC^n.
  \end{equation}
\end{proposition}

\begin{proof}
  First of all, since $\cov^\bbC(z)=\cov^\bbC(z,z)$, one checks that $\calG_z(z)=0$ almost surely, i.e, $\calG_z\in \scrF_z$ almost surely. In fact we have 
  \begin{equation}
    \calG_z=\pi_z(\Exp^\bbC)
  \end{equation}
  where $\pi_z:\scrF\to\scrF_z$ is the orthogonal projection. Indeed, as in the real case, one sees that $\left(\cov(z,\cdot),f\right)=f(z)$ for any $f\in\scrF$, i.e. $\cov^\bbC(z,\cdot)$ is the vector representing the complex linear form $ev^\bbC_z:f\mapsto f(z)$. Moreover one sees that $|ev^\bbC_z|^2=\cov^\bbC(z)$ and thus $\pi_z(f):=f-\frac{f(z)}{\cov^\bbC(z)}\cov^\bbC(z,\cdot)\randin \scrF_z$
  as claimed. The rest of the proof is straightforward:
  \begin{align}
    \frac{1}{\pi}\int_{\scrF_z} F(f) \frac{e^{|f|^2}}{\pi^{n-1}}\, \mrd f 
    =\frac{1}{\pi}\int_{\scrF} F(\pi_z(f)) \frac{e^{|f|^2}}{\pi^{n}}\, \mrd f 
    =\frac{1}{\pi}\bbE\left[F\left(\pi_z(\Exp^\bbC)\right) \right]=\frac{1}{\pi}\bbE\left[F\left(\calG_z\right) \right]
  \end{align}
  which is what we wanted.
\end{proof}

Considering $\scrF$ as a \emph{real} vector space, we have two \emph{real} linear forms:
\begin{equation}
ev^R_z:f\mapsto \Re(f(z)) \mathand ev^I_z:f\mapsto \Im(f(z)).
\end{equation}
We need to compute $\|ev^R_z\wedge ev^I_z\|$ where $\wedge$ denotes the \emph{real} wedge product.

\begin{proposition}\label{prop:normevRevI}
  For all $z\in\bbC^n$, we have $\|ev^R_z\wedge ev^I_z\|=\cov^\bbC(z)$.
\end{proposition}

\begin{proof}
  Let us consider $\{\chi_a,\bar{\chi}_a\}_{a\in A}\subset \Hom_\bbR(\scrF,\bbR)$ the dual basis of the (real orthonormal) basis $\{f^\bbC_a, if^\bbC_a\}_{a\in A}$. Moreover, let us consider the endomorphism $J:\Hom_\bbR(\scrF,\bbR)\to \Hom_\bbR(\scrF,\bbR)$ given for all $\chi\in \Hom_\bbR(\scrF,\bbR)$ and $f\in \scrF$ by 
  \begin{equation}
    J\chi(f):=\chi(if).
  \end{equation}
  Then, we have $\bar{\chi}_a=-J\chi_a$ and $J^2=-1$. Evaluating on basis vector, one can see that 
  \begin{equation}
  ev^R_z=\sum_{a\in A}\Re\left(f^\bbC_a(z)\right)\chi_a-\Im\left(f^\bbC_a(z)\right)\bar{\chi}_a
  \end{equation}
  and that $ev^I_z=-J ev^R_z$. Since the complex structure $J$ preserves the scalar product we get
  \begin{equation}
    \|ev^R_z\wedge ev_z^I\|=\|ev^R_z\wedge J ev_z^R\|=\|ev_z^R\|^2=\sum_{a\in A}\Re\left(f^\bbC_a(z)\right)^2+\Im\left(f^\bbC_a(z)\right)^2=\cov^\bbC(z)
  \end{equation}
  which is what we wanted.
\end{proof}

We will also need the following computation.

\begin{proposition}\label{prop:CdettoRdet}
  Let $\alpha_1,\ldots,\alpha_n\in\Hom_{\bbC}(\bbC^n, \bbC)$, and write $\alpha_j=u_j+i v_j$ with $u_j,v_j\in \Hom_\bbR(\bbC^n,\bbR)$. We have
  \begin{equation}
    |u_1\wedge v_1\wedge \cdots \wedge u_n\wedge v_n|=|\alpha_1\wedge_\bbC\cdots\wedge_\bbC \alpha_n|^2
  \end{equation}
\end{proposition}

\begin{proof}
  As in the previous proof, we consider the complex structure $J:\Hom_\bbR(\bbC^n,\bbR)\to \Hom_\bbR(\bbC^n,\bbR)$ given for all $u\in \Hom_\bbR(\bbC^n,\bbR)$ by $J u:=u\circ i$. Because of the Cauchy-Riemann relations, or equivalently from the equation $\alpha_j\circ i= i\alpha_j$ we have that for all $j=1,\ldots,n$:
  \begin{equation}\label{eq:CR}
    v_j=-Ju_j.
  \end{equation}
  In fact, this turns $\Hom_\bbR(\bbC^n,\bbR)$ into a \emph{complex} vector space and the application \emph{real part} $\Re:\Hom_\bbC(\bbC^n,\bbC)\to \Hom_\bbR(\bbC^n,\bbR)$ is an isomorphism of complex vector space that preserves the Hermitian structure. Indeed, it is imediate from the $\bbC$-linearity that $\Re\circ i=J\circ \Re$ and from the Cauchy Riemann relations \cref{eq:CR}, we get $\Re^{-1}(u)= u-i Ju$. Thus we can work completely in the space $\Hom_\bbR(\bbC^n,\bbR)$ and in particular we have 
  \begin{equation}
    |\alpha_1\wedge_\bbC\cdots\wedge_\bbC \alpha_n|^2=|u_1\wedge_\bbC\cdots\wedge_\bbC u_n|^2
  \end{equation}
  where in the right hand side, we mean the complex determinant when $\Hom_\bbR(\bbC^n,\bbR)$ is considered as a complex vector space with the complex structure $J$. Now it is a classical result that the modulus square of a complex matrix, is the determinant of the underlying real matrix, see for example \cite[Lemma~5]{Basu2016RandomFA}. In our case, this gives:
  \begin{equation}
    |u_1\wedge_\bbC\cdots\wedge_\bbC u_n|^2=|u_1\wedge\cdots\wedge u_n\wedge Ju_1\wedge\cdots\wedge Ju_n|
  \end{equation}
  which, by \cref{eq:CR}, gives the result.
\end{proof}

One more computation.

\begin{proposition}\label{prop:CdetGauss}
  Let $X\randin \bbC^n$ be a (centered) complex Gaussian vector with complex covariance $\Sigma:=\bbE\left[ X \, \overline{X}^T\right].$ Let $X_1,\ldots, X_n$ be iid copies of $X$. We have 
    $\bbE\left[|X_1\wedge_\bbC\cdots \wedge_\bbC X_n|^2\right]=n! \, \det(\Sigma)$
  where on the right hand side the determinant is taken over $\bbC$.
\end{proposition}

\begin{proof}
  First of all, note that the matrix $\Sigma$ is Hermitian, thus $\det(\Sigma)$ is real. Now if $X=M(\xi)$ where $M:\bbC^n\to \bbC^n$ is a $\bbC$-linear map and $\xi\randin \bbC^n$, then $\Sigma=M\overline{M}^T$ and thus $\det(\Sigma)=|\det(M)|^2$. It follows that 
  \begin{equation}\label{eq:thiseqeq}
    \bbE\left[|X_1\wedge_\bbC\cdots \wedge_\bbC X_n|^2\right]= \alpha_n \, \det(\Sigma)
  \end{equation}
  where
  \begin{equation}
    \alpha_n=\bbE\left[|\xi_1\wedge_\bbC\cdots \wedge_\bbC \xi_n|^2\right]
  \end{equation} 
  with $\xi_1,\ldots,\xi_n\randin \bbC^n$ are iid standard complex Gaussian vectors. To compute $\alpha_n$, we use the independence to write 
  \begin{align}
    \alpha_n&=\bbE_{\xi_1}\left[\bbE_{\xi_2\cdots\xi_n}\left[|\xi_1\wedge_\bbC\cdots \wedge_\bbC \xi_n|^2\right]\right]\\
    &=\bbE_{\xi_1}\left[\|\xi_1\|^2\bbE_{\xi_2\cdots\xi_n}\left[|e_1\wedge_\bbC N(\xi_2)\wedge_\bbC\cdots \wedge_\bbC N(\xi_n)|^2\right]\right] \\
    &=\bbE_{\xi_1}\left[\|\xi_1\|^2\bbE_{\xi_2\cdots\xi_n}\left[|e_1\wedge_\bbC \xi_2\wedge_\bbC\cdots \wedge_\bbC \xi_n|^2\right]\right] \\
    &=\bbE_{\xi_1}\left[\|\xi_1\|^2 \right] \alpha_{n-1}=n \alpha_{n-1}
  \end{align}
  where $e_1$ is the first vector of the standard basis of $\bbC^n$, $N:\bbC^n \to \bbC^n$ is a unitary matrix that is such that $N(\xi_1)=\|\xi_1\| e_1$ (note that inside the first expectation $\xi_1$ is fixed) and where, in the third equality, we used the unitary invariance of the standard complex Gaussian vectors. Now it is enough to see that $\alpha_1=1$ by definition of standard complex Gaussians and thus $\alpha_n=n!$. Reintroducing in \cref{eq:thiseqeq} gives the result.
\end{proof}

We are now ready to give a proof of the main result.

\begin{proof}[Proof of \cref{thm:BKKprob}]
  According to the general theory, that is to \cref{prop:mainZKROK}, we have 
  \begin{equation}\label{eq:inproofBKKmain}
    \esol{\Exp^\bbC;U}=\int_U \delta(z) \, \mrd z
  \end{equation}
  where the Kac-Rice density $\delta$ is given by 
  \begin{equation}
    \delta(z):=\frac{1}{\|ev_z^R\wedge ev_z^I\|^n}\int_{(\scrF_z)^n}|D_zf_1^R\wedge D_zf_1^I\wedge\cdots\wedge D_zf_n^R\wedge D_zf_n^I| \rho_{\Exp^\bbC}(f_1)\cdots \rho_{\Exp^\bbC}(f_n)\, \mrd f_1\cdots \mrd f_n
  \end{equation}
  where $f^R$, respectively $f^I$, denotes the real, respectively imaginary, part of $f$. The coefficient in front of the integral was computed in \cref{prop:normevRevI}.
  Now let us write $F:(\scrF_z)^n\to \bbR$ given by
  \begin{equation}
    F(f_1,\ldots,f_n):= |D_zf_1^R\wedge D_zf_1^I\wedge\cdots\wedge D_zf_n^R\wedge D_zf_n^I|.
  \end{equation}
  Then, using \cref{prop:Gz}, we can replace the integrals over $\scrF_z$ by the expectation of the random field $\calG_z$. More precisely, we have:
  \begin{equation}\label{eq:inproofdeltais EF}
  \delta(z)=\frac{1}{\cov^\bbC(z)^n \pi^n }\bbE\left[F(\calG_z^1,\ldots,\calG_z^n)\right]
  \end{equation}
  where $\calG_z^1,\ldots,\calG_z^n\randin \scrF_z$ are iid copies of $\calG_z$ defined in \cref{eq:defGz}.

  Let us denote $\del_z:\scrF\to \Hom_{\bbC}(\bbC^n,\bbC)$ the complex differential, i.e, $\del_zf=D_zf^R+iD_zf^I$. Note that we use here that the elements of $\scrF$ consist of holomorphic functions. We can now use \cref{prop:CdettoRdet} to switch from the real determinant, that is from $F$, to a complex determinant. Namely we have:
  \begin{equation}
    F(f_1,\ldots,f_n)=|\del_zf_1\wedge_\bbC\cdots\wedge_\bbC \del_zf_n|^2.
  \end{equation}

  Since $\del_z$ is a linear map and $\calG_z\randin \scrF_z$ is a complex Gaussian vector, it follows that $\del_z G_z$ is a Gaussian vector. Thus we can use \cref{prop:CdetGauss} to get
  \begin{equation}\label{eq:inproofFisdet}
    \bbE\left[F(\calG_z^1,\ldots,\calG_z^n)\right]= n! \det(\Sigma_z)
  \end{equation}
  where $\Sigma_z$ is the complex covariance of $\calG_z$, i.e.,$\Sigma_z=\bbE\left[\del_z \calG_z \overline{\del_z \calG_z}^T \right].$
  From its definition in \cref{eq:defGz}, $\calG_z$ is given for all $z'\in \bbC $ by:
  \begin{equation}
    \calG_z(z')=\sum_{a\in A}\xi_a \left(f^\bbC_a(z')-f_a^\bbC(z)\frac{\cov^\bbC(z,z')}{\cov^\bbC(z)}\right).
  \end{equation}
  We need to apply $\del_z$ to this expression, that is, to differentiate in $z'$. We first note that $\del_z f^\bbC_a=f^\bbC_a(z)\,a.$
  where $a\in A$ is identified with the $\bbC$-linear map $(a,\cdot)$. Then, we compute (recall that we only differentiate in $z'$):
  \begin{equation}
    \del_z \cov^\bbC(z,\cdot)=\sum_{a\in A} \overline{f^\bbC_a(z)}\, \del_zf^\bbC_a= \sum_{a\in A} \overline{f^\bbC_a(z)}\, f^\bbC_a(z)\, a= \cov^\bbC(z)\, \mu^\bbC(z).
  \end{equation} 
  It follows that 
  \begin{equation}
    \del_z\calG_z=\sum_{a\in A}\xi_a f^\bbC_a(z)\left(a-\mu^\bbC(z)\right)
  \end{equation}
  which yields
  \begin{equation}\label{eq:Sigmaz}
    \Sigma_z= \sum_{a\in A} |f^\bbC_a(z)|^2\left(a-\mu^\bbC(z)\right)\overline{\left(a-\mu^\bbC(z)\right)}^T.
  \end{equation}
  Let us take a moment to gather what we obtained so far. Combining \cref{eq:inproofdeltais EF} and \cref{eq:inproofFisdet}, we obtain an expression for the Kac Rice density that we can reintroduce in the main equation \cref{eq:inproofBKKmain}. We get: 
  \begin{equation}\label{eq:inproofBkkinterlude}
    \esol{\Exp^\bbC;U}=\frac{n!}{\pi^n}\int_U \frac{\det(\Sigma_z)}{\cov^\bbC(z)^n} \, \mrd z.
  \end{equation}
  To obtain the claimed result \cref{eq:deltaisdeldelbar}, we need to compare $\Sigma_z$ and $\del_z\delbar_z\Phi^\bbC$. Before computing $\del_z\delbar_z\Phi^\bbC$, let us note that 
  \begin{equation}
    \del_z \cov^\bbC=\cov^\bbC(z) \mu^\bbC(z) \mathand \delbar_z \cov^\bbC= \cov^\bbC(z) \overline{\mu^\bbC(z)}^T
  \end{equation}
  It yields:
  \begin{equation}
    \del_z\delbar_z\Phi^\bbC = \frac{\cov^\bbC(z)\del_z\delbar_z\cov^\bbC-\del_z \cov^\bbC\delbar_z \cov^\bbC}{\cov^\bbC(z)^2}=\frac{\del_z\delbar_z\cov^\bbC}{\cov^\bbC(z)}-\mu^\bbC(z)\overline{\mu^\bbC(z)}^T.
  \end{equation}
  Moreover, we have 
  \begin{equation}
    \del_z\delbar_z\cov^\bbC=\sum_{a\in A} |f^\bbC_a(z)|^2 a \overline{a}^T.
  \end{equation}
  We now expand \cref{eq:Sigmaz} to obtain:
  \begin{align}
    \Sigma_z&=\sum_{a\in A}|f^\bbC_a(z)|^2  a \,\overline{a}^T+\cov^\bbC(z) \, \mu^\bbC(z)\overline{\mu^\bbC(z)}^T -\sum_{a\in A}|f^\bbC_a(z)|^2 \, \left(a\, \overline{\mu^\bbC(z)}^T- \mu^\bbC(z)\, \overline{a}^T\right) \\
    &=\sum_{a\in A}|f^\bbC_a(z)|^2  a\, \overline{a}^T+\cov^\bbC(z) \, \mu^\bbC(z)\overline{\mu^\bbC(z)}^T- 2 \cov^\bbC(z)\,  \mu^\bbC(z)\overline{\mu^\bbC(z)}^T \\
    &=\sum_{a\in A}|f^\bbC_a(z)|^2  a\, \overline{a}^T- \cov^\bbC(z)\,  \mu^\bbC(z)\overline{\mu^\bbC(z)}^T \\
    &= \cov^\bbC(z)\, \del_z\delbar_z\Phi^\bbC.
  \end{align}
  Reintrocucing in \cref{eq:inproofBkkinterlude} gives the desired \cref{eq:deltaisdeldelbar}.

  To obtain the other statements in the Theorem, one needs to note that, when $A\subset \bbR^n$, comparing with \cref{eq:gx and alpha}, we have that $\del_z\delbar_z\Phi^\bbC=\tfrac{1}{2} D^2_z\Phi^\bbC$. Moreover, $\cov^\bbC$ is invariant by translation in the imaginary directions, i.e., for every real number $r\in \bbR$, we have $\cov^\bbC(z+ir)=\cov^\bbC(z)$. It follows that the Kac-Rice density has the same property and we obtain
  \begin{equation}
    \esol{\Exp^\bbC;V\times i (-\pi,\pi)}=\frac{n!}{(2\pi)^n}\int_{V\times i (-\pi,\pi)}\det\left(D^2_z\Phi^\bbC\right)\, \mrd z=n! \int_{V}\det\left(D^2_x\Phi^\bbC\right)\, \mrd x.
  \end{equation}
  Now we find ourselves in the real situation again and applying the change of variable $\mu^\bbC=\mu: V\mapsto \mu(V)\subset\interior(P)$ gives the result.
\end{proof}

\bibliographystyle{alpha}
\bibliography{literature}
\end{document}